\documentclass{amsart}
\usepackage{amssymb}
\usepackage{amscd}
\usepackage{amsmath}
\usepackage{enumerate}
\usepackage{float}
\usepackage[T1]{fontenc} 
\usepackage[english]{babel} 
\usepackage{amsmath,amsfonts,amsthm} 
\usepackage{fancyhdr} 
\usepackage{verbatim} 
\usepackage[hyperindex,breaklinks]{hyperref} 
\hypersetup{          
    colorlinks=true,
    linkcolor=blue,
    filecolor=magenta,      
    urlcolor=blue,
    citecolor=magenta,
    linktoc=all,
}

\usepackage[a4paper,bindingoffset=0.2in,%
left=1in,right=1in,top=1in,bottom=1.3in,%
footskip=.25in]{geometry}
\usepackage[nameinlink]{cleveref}

\PassOptionsToPackage{linktocpage}{hyperref}

\usepackage{graphicx} 
\usepackage[font=footnotesize,labelfont=bf]{caption} 
\usepackage{tikz-cd} 
\usepackage{setspace}

\numberwithin{equation}{section} 
\numberwithin{figure}{section} 
\numberwithin{table}{section} 

\theoremstyle{definition}
\newtheorem{remark}[equation]{Remark}

\crefname{notation}{notation}{notation}
\newtheorem{example}[equation]{Example}

\theoremstyle{plain}
\newtheorem{definition}[equation]{Definition}
\newtheorem{theorem}[equation]{Theorem}
\newtheorem{lemma}[equation]{Lemma}
\newtheorem{corollary}[equation]{Corollary}
\newtheorem{proposition}[equation]{Proposition}

\newtheorem{thmA}{Theorem}

\DeclareMathOperator{\lcm}{lcm} 

\newcommand{\R}{\mathbb{R}}

\newcommand{\Z}{\mathbb{Z}}

\newcommand{\Si}{\Sigma}

\newcommand{\G}{\Gamma}
\newcommand{\tat}{t\^ete-\`a-t\^ete }
\newcommand{\Tat}{T\^ete-\`a-t\^ete }

\newcommand{\MS}{\mathbb{S}}
\newcommand{\TG}{\widetilde{\G}}

\newcommand{\calC}{\mathcal{C}}

\newcommand{\MCG}{\mathrm{MCG}}

\usepackage{color}

\begin{document}
    \title{General T\^ETE-\`A-T\^ETE graphs and Seifert 
    manifolds}
    
    \author{Pablo Portilla Cuadrado}
    \address{(1)ICMAT,  Campus Cantoblanco UAM, C/ Nicol\'as Cabrera, 13-15, 
    28049 Madrid, Spain
        (2) BCAM,  Basque Center for Applied Mathematics, Mazarredo 14, 48009 
        Bilbao, Basque Country,  Spain}
    \email{p.portilla89@gmail.com}
    \begin{abstract}
         \Tat graphs and relative \tat graphs were introduced by N. A'Campo in 
         2010 to model monodromies of isolated plane curves. By recent work
         of  Fdez de Bobadilla,  Pe Pereira and the author, they provide a 
         way of modeling the periodic mapping classes that leave some boundary 
         component invariant. In this work we introduce the notion of general 
         \tat graph and prove that they model all periodic mapping classes. We 
         also describe algorithms that take a Seifert manifold and a horizontal 
         surface and return a \tat graph and vice versa.
    \end{abstract}
    \thanks{Author supported by SVP-2013-067644 Severo Ochoa FPI grant and by 
    project by MTM2013-45710-C2-2-P, the two of them by MINECO; also supported 
    by the project ERCEA 615655 NMST Consolidator Grant}
    \maketitle
    
    \tableofcontents

    \section{Introduction}
    
    In \cite{Camp1} N. A'Campo introduced the notion of pure \tat graph in 
    order to model monodromies of plane curves. These 
    are metric ribbon graphs $\G$ without univalent vertices that satisfy a 
    special  property called {the \tat property}. One usually sees the ribbon 
    graph as a strong deformation retract of a surface $\Si$ with non-empty 
    boundary, which is called the {\em thickening}. The \tat 
    propety says that if you pick a point $p$, then walk distance of $\pi$ in 
    any direction from that point and you always turn right at vertices, 
    you get to the same point no matter the inital direction. This 
    property defines an  element in the mapping class group $\MCG^+(\Si, 
    \partial \Si)$ which is freely periodic. 

    In \cite{Graf}, C. Graf proved  that if one allows univalent vertices in 
    \tat graphs, then the set of mapping classes produced by \tat graphs are 
    all freely   periodic mapping classes of $\MCG^+(\Si, \partial \Si)$  with 
    positive  fractional Dehn twist coefficients. In \cite{jmp} this result was 
    improved  by showing that one  does not need to enlarge the original class 
    of metric  ribbon graphs used  to prove the same theorem.

    A bigger class of graphs was introduced in \cite{Camp1}, the relative \tat 
    graphs. These are pairs $(\G, A)$ formed by a metric ribbon graph $\G$ and 
    a subset $A \subset \G$ which is a collection of circles. Seen as a strong 
    deformation retract of a surface, this pair is properly embedded, i.e. 
    $(\G, A) \hookrightarrow (\Si, \partial \Si)$ and $\partial \Si \setminus  
    A \neq \emptyset$. They satisfy the {\em relative \tat property} which is 
    similar to the \tat property and defines an element in $\MCG^+(\Si, 
    \partial \Si \setminus A)$ which is freely periodic. In \cite{jmp} it was 
    proved that the set of mapping classes modeled by relative \tat graphs are 
    all freely periodic mapping classes of $\MCG^+(\Si, \partial \Si \setminus 
    A)$  with positive  fractional Dehn twist coefficients at the boundary 
    components in $\partial \Si \setminus A$.
        
    At this point there is a natural question which was already posed in 
    \cite{Graf1}, 
    {\em how can one complement the definition of \tat graph to be able to 
    model all periodic mapping classes?} (even if they do not leave any 
    boundary component invariant). To cover these cases, we introduce general 
    \tat 
    graphs (see \Cref{def:general_tat}). These are metric ribbon graphs with 
    some special univalent vertices 
    $\mathcal{P} \subset \partial \Si$ and a permutation acting on these 
    vertices. An analogous general \tat property is defined. A 
    general \tat graph defines a periodic mapping class in $\MCG^+(\Si)$.
    
    As our main result, we prove:
    
     \begin{thmA}
    The mapping class of any  periodic automorphism $\phi: \Si \to \Si$ of a 
    surface can be realized via general \tat graphs. Moreover, the general \tat 
    graph can be extended to a pure or relative \tat graph, thus realizing the 
    automorphism as a restriction to $\Si$ of a periodic automorphism on a 
    surface $\hat\Si \supset \Si$ that leaves each boundary component invariant.
     \end{thmA}

    The mapping torus of a periodic surface automorphism is a Seifert manifold 
    and a orientable horizontal surface of a fiber-oriented Seifert manifold 
    has a periodic monodromy induced on it. Hence, it is natural to assign a 
    \tat graph to a Seifert manifold and a horizontal surface on 
    it and vice versa. The rest of the work is devoted to understanding this 
    relation.
    
    In \Cref{sec:seifert_man} we briefly review the theory of Seifert manifolds 
    and plumbing graphs. The theory of Seifert manifolds is classical and there 
    is plenty  of literature about it (see for example \cite{Neum}, 
    \cite{Neum2}, \cite{Neum3}, \cite{Neum5}, \cite{Hir}, \cite{Hatch} or 
    \cite{Peder}). Because of this,  we try to avoid repeating well-known 
    results.  However,  there is not such thing as {\em standard conventions} 
    in Seifert  manifolds. Since the conventions that we choose are very 
    important for \Cref{sec:algor}, we take  some time to fix them carefully. 
    
    In \Cref{sec:hor_sur}, we review the theory about horizontal surfaces in a 
    Seifert manifold $M$. This has been studied from different 
    point of views in the literature. For example, in \cite{EisNeu} it is 
    proved a classification in the more general case when $M$ is an integral 
    homology sphere. In \cite{Pich} Pichon provides existence of fibrations of 
    any graph manifold $M$ by producing algorithmically a complete list of 
    the conjugacy and isotopy invariants of the automorphisms  whose 
    associated  mapping torus is  diffeomorphic to $M$. We review
    some of this results and write them in a language that best suits our 
    notation and conventions. Among these results is a classification of 
    horizontal surfaces of a Seifert manifold with boundary.
    
    In \Cref{sec:algor} we detail two algorithms. One takes 
    a Seifert  manifold and a horizontal surface as input and returns as 
    output a general,  relative or pure \tat graph realizing the horizontal 
    surface and its monodromy. This algorithms differs from  similar results 
    in the literature in that our method produces directly the monodromy (in 
    this case the \tat graph) without computing the conjugacy and isotoy 
    invariants of the corresponding periodic mapping class. The other algorithm 
    works in the opposite direction by taking a general, relative or pure \tat 
    graph and producing the corresponding Seifert manifold and horizontal 
    surface.
    
    Finnaly, \Cref{sec:examples} contains a couple of detailed examples in 
    which we apply the two  algorithms.
    
    \section*{Acknowledgments}
    I thank my advisors Javier Fdez. de Bobadilla and Mar\'ia Pe Pereira for 
    many useful comments and suggestions during the writing of this work. I am 
    specially grateful to Enrique Artal Bartolo for taking the 
    time to explain carefully many of the aspects of the theory of Seifert 
    manifolds. 
    
    I also thank BCAM for having the ideal environment in which most of this 
    work was developed.
    
    \section{General \tat structures}
    \label{sec:general}

    In this section we study any orientation preserving periodic homeomorphism. 
    Let $\phi: \Si \to \Si$ be such a homeomorphism. We realize its 
    boundary-free isotopy 
    type and its conjugacy class in $MCG(\Si)$ by a generalization of \tat 
    graphs, using a 
    technique that reduces to the case of homeomorphisms of a larger surface 
    that leave all boundary components invariant.
    
    Contrarily to what was done in \cite{jmp}, we allow ribbon graphs with some 
    {\em special} univalent vertice.
    
    \begin{definition}
        A ribbon graph with boundary is a pair $(\G,\mathcal{P})$ where $\G$ is 
        a ribbon
        graph, and $\mathcal{P}$ is the set of univalent vertices, with the 
        following additional property: given any vertex $v$ of valency greater 
        than $1$ in the cyclic ordering of adjacent edges 
        $e(v)$ there are no two consecutive edges connecting $v$ with vertices 
        in $\mathcal{P}$.
    \end{definition}
    
    In order to define the thickening of a ribbon graph with boundary we need 
    the following construction:
    
    Let $\G'$ be a ribbon graph (without univalent vertices)and let $\Sigma$ be 
    its thickening. Let 
    $$g_{\G'}:\Si_{\G'}\to\Si$$
    be the gluing map. The surface $\Si_{\G'}$ splits as a disjoint union of
    cylinders $\coprod_i\widetilde{\Si}_i$. Let $w$ be a vertex of $\G'$. The
    cylinders  $\widetilde{\Si}_i$ such that $w$ belongs to 
    $g_{\G'}(\widetilde{\Si}_i)$ 
    are in  a natural bijection with the pairs of consecutive edges $(e',e'')$ 
    in the  cyclic order of the set $e(w)$ of adjacent edges to $w$.
    
    Let $(\G,\mathcal{P})$ be a ribbon graph with boundary. The graph $\G'$ 
    obtained by erasing from $\G$ the set $E$ of all vertices in $\mathcal{P}$ 
    and its adjacent  edges is a ribbon graph. Consider the thickening surface 
    $\Si$ of $\G'$.  Let $e$ be an edge connecting a vertex $v\in\mathcal{P}$ 
    with another vertex $w$, let $e'$ and $e''$ be the inmediate predecesor and 
    succesor of $e$ in the cyclic order of $e(w)$. By the defining property of 
    ribbon graphs with boundary  they are consecutive edges in $e(w)\setminus 
    E$, and hence determine a unique associated cylinder which will be denoted 
    by  $\widetilde{\Si}_i(v)$.
    
    Each cylinder $\widetilde{\Si}_i$ has two boundary components, one, denoted 
    by $\widetilde{\G}_i$ corresponds to the boundary component obtained by 
    cutting the graph,  and the other, called $C_i$, corresponds to a boundary 
    component of $\Si$. Fix a cylinder $\widetilde{\Si}_i$. Let 
    $\{v_1,...,v_k\}$ be the vertices of $\mathcal{P}$ whose associated 
    cylinder is $\widetilde{\Si}_i$. Let $\{e_1,...,e_k\}$ be the
    corresponding edges, let $\{w_1,...,w_k\}$ be the corresponding vertices at
    $\G'$, and let 
    $\{w'_1,...,w'_k\}$ be the set of preimages by $g_{\G'}$ contained in
    $\widetilde{\Si}_i$. The defining property of ribbon graphs with boundary 
    imply that $w'_i$ and $w'_j$ 
    are pairwise different if $i\neq j$. Furthermore, since $\{w'_1,...,w'_k\}$ 
    is included in the circle $\widetilde{\G}_i$, which has an orientation 
    inherited from $\Sigma$, 
    the set $\{w'_1,...,w'_k\}$, and hence also $\{e_1,...,e_k\}$ and
    $\{v_1,...,v_k\}$ has a cyclic order. We assume that our indexing respects 
    it. 
    
    Fix a product structure $\MS^1\times I$ for each cylinder 
    $\widetilde{\Si}_i$, where $\MS^1\times\{0\}$ corresponds to the boundary 
    component $\widetilde{\G}_i$, and $\MS^1\times\{1\}$ corresponds to the 
    boundary component of $C_i$.
    
    Using this product structure we can embedd $\G$ in $\Si$: for each vertex
    $v\in\mathcal{P}$ consider the corresponding cylinder 
    $\widetilde{\Si}_{i(v)}$,
    let $w'$ be the point in $\widetilde{\G}_{i(v)}$ determined above. We 
    embedd the segment $g_{\G'}({w'}\times I)$ in $\Si$. 
    
    Doing this for any vertex $v$ we obtain an embedding of $\G$ in $\Si$ such 
    that all the vertices $\mathcal{P}$ belong to the boundary $\partial\Si$, 
    and such that $\Si$ admits $\G$ as a regular deformation retract. 
    
    \begin{definition}
        Let $(\G,\mathcal{P})$ be a ribbon graph with boundary. We define the 
        thickening
        surface $\Si$ of $(\G,\mathcal{P})$
        to be the thickening surface of $\G'$ toghether with the embeding
        $(\G,\mathcal{P})\subset (\Si,\partial\Si)$ constructed above. We say 
        that $(\G,\mathcal{P})$ is a \emph{general spine} of 
        $(\Si,\partial\Si)$.
    \end{definition}
    
    \begin{definition}[General safe walk]
        Let $(\G,\mathcal{P})$ be a metric ribbon graph with boundary. Let 
        $\sigma$ be apermutation of $\mathcal{P}$.
        
        We define a \emph{general safe walk in $(\G,\mathcal{P},\sigma)$} 
        starting at a
        point $p \in \G \setminus v(\G)$ to be a map $\gamma_p:[0,\pi] 
        \rightarrow \G$
        such that
        \begin{itemize}
            \item[1)] $\gamma_p(0)=p$ and $|\gamma_p'|=1$ at all times.
            \item[2)] when $\gamma_p$ gets to a vertex of valency $\geq 2$ it 
            continues
            along the next edge in the cyclic order.
            \item[3)] when $\gamma$ gets to a vertex in $\mathcal{P}$, it 
            continues along
            the edge indicated by the permutation $\sigma$. 
        \end{itemize}
    \end{definition}
    
    \begin{definition}[General \tat graph]\label{def:general_tat}
        Let $(\G,\mathcal{P},\sigma)$ be as in the previous definition. Let 
        $\gamma_p$,
        $\omega_p$ be the two safe walks starting at a point $p$ in 
        $\G\setminus v(\G)$.
        
        We say $\G$ has the \emph{general \tat property} if 
        \begin{itemize}
            \item for any $p\in \G \setminus v(\G)$ we have 
            $\gamma_p(\pi)=\omega_p(\pi)$
        \end{itemize}
        Moreover we say that $(\G,\mathcal{P},\sigma)$ gives a \emph{general 
        \tat
            structure} for $(\Si,\partial\Si)$ if $(\Si,\partial\Si)$ is the 
            thickening of
        $(\G,\mathcal{P})$. 
    \end{definition}
    
    In the following construction we associate to a general \tat graph
    $(\G,\mathcal{P},\sigma)$ a homeomorphism of $(\G,\mathcal{P})$ which 
    restricts to the permutation
    $\sigma$ in $\mathcal{P}$; we call it the \emph{general \tat homeomorphism 
    of $(\G,\mathcal{P},\sigma)$}. We construct also a homeomorphism of the 
    thickening surface which leaves $\G$ invariant and restricts on $\G$ to 
    the general \tat homeomorphism of $(\G,\mathcal{P},\sigma)$. We construct 
    the homeomorphism on the graph and on its thickening simultaneously.
    
    Consider the homeomorphism of $\G'\setminus v(\G)$ defined by
    $$p\mapsto \gamma_p(\pi).$$
    The same proof of \cite[Lemma 3.6]{jmp} shows that there is an extension of 
    this homeomorphism to a homeomorphism 
    $$\sigma_{\G}:\G\to\G.$$ The restriction of the general \tat homeomorphism 
    that we are constructing  to $\G'$ coincides with $\sigma_{\G}$. The 
    mapping $\sigma_\G$ leaves $\G'$
    invariant for being a homeomorphism. Let $\TG'$ be the union of the circles 
    $\TG_i$. The homeomorphism $\sigma_\G|_{\G'}$ 
    lifts to a periodic homeomorphism 
    $$\tilde{\sigma}:g_{\G'}^{-1}(\TG')\to 
    g_{\G'}^{-1}(\TG'),$$
    which may exchange circles in the following way. For any 
    $p\in\TG'$,
    the points in $g_{\G'}^{-1}(p)$ corresponds to the starting poing of safe 
    walks in $\TG'$ starting at $p$. A safe walk starting at $p$ is 
    determined by the point $p$ and an starting direction at an edge containg 
    $p$. 
    
    As we have seen, if $(\G,\mathcal{P},\sigma)$ is a general \tat structure 
    for $(\Si,\partial\Si)$ then the surface $\Si_{\G'}$ is a disjoint union of
    cylinders. The lifting $\tilde{\sigma}$ extends to $\Si_{\G'}$ similarly as 
    with the definition of the homoemorphism corresponding to a \tat structure 
    defined in \cite{jmp}. This extension interchanges some cylinders 
    $\widetilde{\Si}_i$ and goes down to an homeomorphism of $\Si$. We denote 
    it by $\phi_{(\G,\mathcal{P},\sigma)}$. If
    necessary, we change the embedding of the part of $\G$ not contained in 
    $\G'$ in
    $\Si$ 
    such that it is invariant by $\phi_{(\G,\mathcal{P},\sigma)}$. This is done 
    by
    an adequate choice of the trivilizations of the cylinders.
    
    \begin{definition}
        \label{def:generaltatauto}
        The homeomorphism  
        $\phi_{(\G,\mathcal{P},\sigma)}$ is by definition the homeomorphism of 
        the
        thickening, and its restriction to $\G$ the \emph{general \tat 
        homeomorphism of 
            $(\G,\mathcal{P},\sigma)$}.
    \end{definition}
    
    With the notation and definitions introduced we are ready to state and 
    proof the main result of the work.
    
    \begin{theorem}\label{thm:general_tat} 
        Given a periodic homeomorphism $\phi$ of a surface with boundary
        $(\Si,\partial\Si)$ which is not a disk or a cylinder, the following 
        assertions
        hold: 
        \begin{enumerate}[(i)]
            \item There is a general \tat graph $(\G,\mathcal{P},\sigma)$ such 
            that the 
            thickening of $(\G,\mathcal{P})$ is $(\Si,\partial\Si)$, the 
            homeomorphism
            $\phi$ leaves $\G$ invariant and we have the 
            equality $\phi|_{\G}=\phi_{(\G,\mathcal{P},\sigma)}|_{\G}$.
            \item We have the equality of boundary-free isotopy classes
            $[\phi|_{\G}]=[\phi_{(\G, \mathcal{P}, \sigma)}]$.
            \item The homeomorphisms $\phi$ and $\phi_{(\G, \mathcal{P}, 
            \sigma)}$ are
            conjugate.
        \end{enumerate}
        
    \end{theorem} 
    
    \begin{proof}
        In the first part of the proof we extend the homeomorphism $\phi$ to a
        homeomorphism $\hat{\phi}$ of a bigger surface $\hat{\Si}$ 
        that
        leaves all the boundary 
        components invariant. Then, we find a \tat graph $\hat{\G}$ for
        $\hat{\phi}$ such that $\hat{\G}\cap \Si$, with a small 
        modification in the metric and a suitable permutation, is a general 
        \tat graph for $\phi$. 
        
        Let $n$ be the order of the homeomorphism. Consider the permutation 
        induced by $\phi$ in the set of boundary 
        components. Let $\{C_1,...,C_m\}$ be an orbit of cardinality strictly 
        bigger than  $1$, numbered such that $\phi(C_i)=C_{i+1}$ and 
        $\phi(C_m)=C_1$. Take an arc $\alpha\subset C_1$ small enough so that 
        it is disjoint from all its   iterations by $\phi$. 
        Define the arcs $\alpha^i:=\phi^i(\alpha)$ for $i\in\{0,...,n-1\}$, 
        which are contained in $\cup_iC_i$. Obviously we have the equalities 
        $\alpha^{i+1}=\phi(\alpha^i)$ and $\phi(\alpha^{n-1})=\alpha^0=\alpha$.
        
        \begin{figure}[!ht]
            \centering
            \includegraphics[scale=0.5]{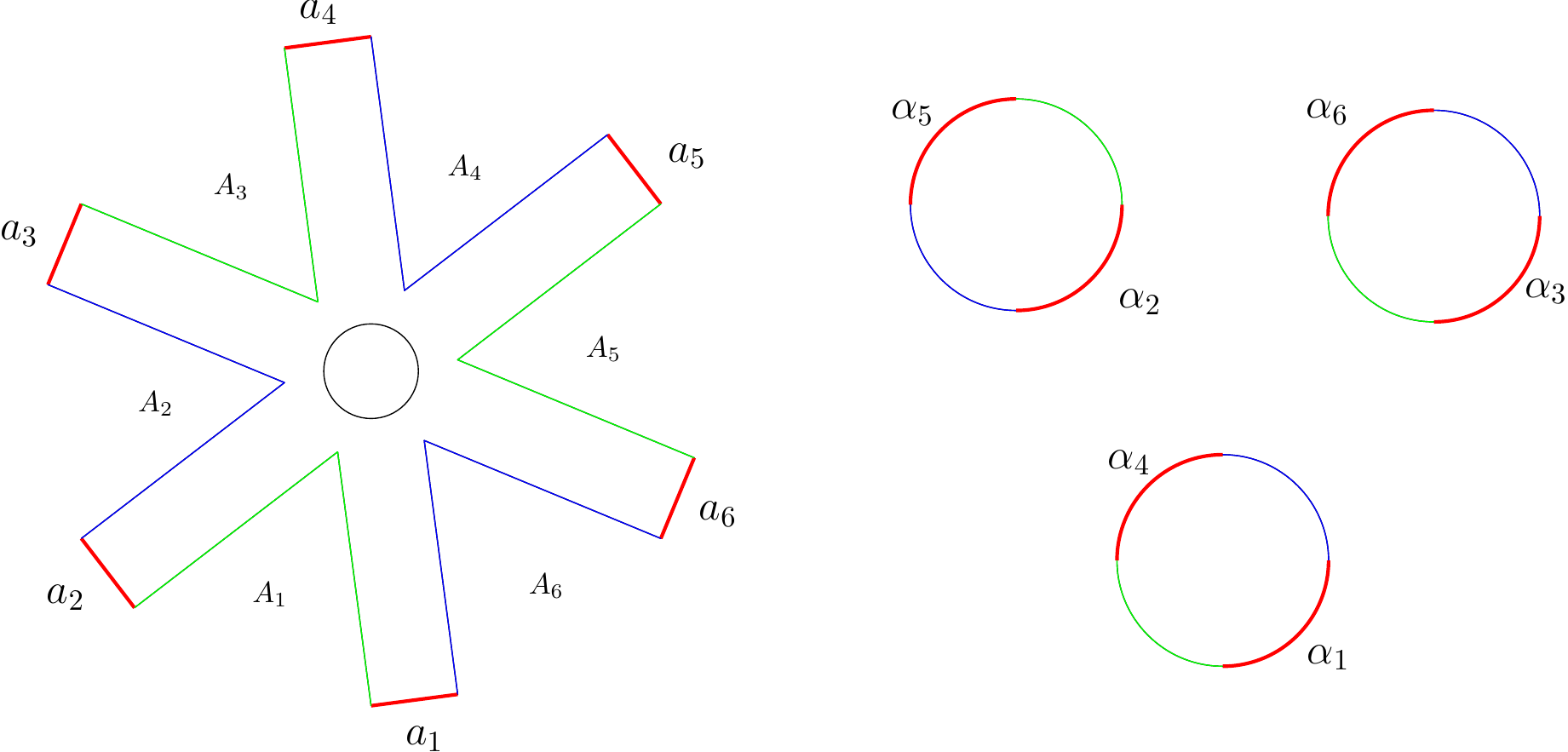}
            \caption{Example of a star-shaped piece $S$ with 6 arms on the left 
            and boundary
                components components on the right. The arcs along which the 
                two pieces are
                glued, are marked in red. In blue and red are the boundaries of 
                the two disks
                that we used to cap off the new boundaries.}
            \label{fig:rot}
        \end{figure}
        
        We consider a star-shaped piece $S$ of $n$ arms as in \Cref{fig:rot}. 
        We denote
        by $D$ the central boundary component. 
        Let $a^0, \ldots, a^{n-1}$  be the boundary of the arms of the 
        star-shaped piece
        labelled in the picture, 
        oriented counterclockwise. We consider the rotation $r$ of order $n$ 
        acting on
        this piece such that $r(a^i)=a^{i+1}$. Note that this rotation leaves 
        $D$
        invariant. 
        
        We consider the surface $\hat{\Si}$ obtained by gluing $\Si$ and $S$ 
        identifying
        $a^i$ with $\alpha^i$ reversing the orientation, and such that $\phi$ 
        and the
        rotation $r$
        glue to a periodic homeomorphism $\hat{\phi}$ in the resulting surface.
        
        The boundary components of the new surface are precisely the boundary 
        components 
        of $\Si$
        different from $\{C_1,...,C_m\}$, the new boundary component $D$, and 
        the boundary components $C'_1$,...,$C'_k$ that contain the part of the 
        $C_i$'s not included in the union $\cup_{i=0}^{n-1}\alpha^i$.
        
        The homeomorphism $\hat{\phi}$ leaves $D$ invariant and may interchange 
        the new boundary components $C'_1$,...,$C'_k$. 
        We cup each component $C'_i$ with a disk $D_i$ and extend the 
        homeomorphism by
        the Alexander trick, obtaining a homeomorphism $\hat{\phi}$ of a 
        bigger surface 
        $\hat{\Si}$. The only new ramification points that the action of
        $\hat{\phi}$ may induce are the centers $t_i$ of these disks. We 
        claim
        that, in fact, 
        each of the $t_i$'s is a ramification point.
        
        Denote the quotient map by 
        $$p:\hat{\Si}\to \hat{\Si}^{\hat{\phi}}.$$
        
        In order to prove the claim notice that the difference
        $\hat{\Si}\setminus\Si$ is homeomorphic to a closed surface with 
        $m+1$
        disks removed. On the other hand 
        the difference of quotient surfaces
        $\hat{\Si}^{\hat{\phi}}\setminus\Si^\phi$ is homeomorphic 
        to a
        cylinder. Since $m$ is strictly bigger than $1$, Hurwitz 
        formula for $p$ forces the existence of ramification points. Since $p$ 
        is a
        Galois cover each $t_i$ is a ramification point.
        
        The new boundary component of $\hat{\Si}^{\hat{\phi}}$ 
        corresponds
        to $p(D)$, where $D$ is invariant by $\hat{\phi}$. The point 
        $q_1:=p(t_i)$
        is then a branch point of $p$.
        
        We do this operation for every orbit of boundary components in $\Si$ of
        cardinality greater than $1$. Then we get a surface $\hat{\Si}$ 
        and an
        extension 
        $\hat{\phi}$ of $\phi$ that leaves all the boundary components 
        invariant.
        The quotient surface $\hat{\Si}^{\hat{\phi}}$
        is obtained from $\Si^\phi$ attaching some cylinders $\calC_j$ to some 
        boundary
        components. Let 
        $$p:\hat{\Si}\to\hat{\Si}^{\hat{\phi}}$$ 
        denote the quotient map. Comparing $p|_{\Si}$ and 
        $p|_{\hat{\Si}}$, we see
        that we have only one new branching point $q_j$ in every cylinder 
        $\calC_j$.
        
        Now we construct a \tat graph for $\hat{\phi}$ modifying slightly 
        the
        construction of \cite[Theorem 5.12]{jmp}.

        To fix ideas we consider the case in which the genus of the quotient
        $\hat{\Si}^{\hat{\phi}}$ is positive. The modification of 
        the genus
        $0$ case is exactly 
        the same. As in \cite[Theorem 5.12]{jmp} we use a planar representation 
        of 
        $\Si^\phi$ as a
        convex $4g$-gon in $\R^2$
        with $r$ disjoint open disks removed from its convex hull and whose 
        edges are
        labelled clockwise like 
        $a_1b_1a_1^{-1}b_1^{-1}a_2b_2a_2^{-1}b_2^{-1} \ldots
        a_{g}b_{g}a_{g}^{-1}b_{g}^{-1}$, we number the boundary components 
        $C_i\subset
        \partial\Si$, $1\leq i\leq r$, we
        denote by $d$ the arc $a_2b_2a_2^{-1}b_2^{-1} \ldots
        a_{g}b_{g}a_{g}^{-1}b_{g}^{-1}$, and we consider $l_1$,...,$l_{r-1}$ 
        arcs as in
        \Cref{fig:gen2}. We denote by 
        $c_1$,...,$c_{r}$ the edges in which $a_1^{-1}$ (and $a_1$) is 
        subdivided
        according to the component $p(C_i)$ they enclose. 
        
        We impose the further condition that each of the regions in which the 
        polygon is
        subdivided by the $l_i$'s encloses not only a component $p(C_i)$, but 
        also the
        branching 
        point $q_i$ that appears in the cylinder $\calC_i$. We assume that the 
        union of
        $d$, $a_1b_1a_1^{-1}b_1^{-1}$ and the $l_i$'s contains all the 
        branching points
        of $p$
        except the $q_i$'s.
        
        In order to be able to lift the retraction we need that the spine that 
        we draw
        in the quotient contains all branching points. In order to achieve this 
        we add
        an edge
        $s_i$ joining $q_i$ and some interior point $q'_i$ of $l_i$ for $i =1, 
        \ldots,
        r-1$ and joining $q_r$ with some interior point $q'_r$ of $l_{r-1}$. We 
        may
        assume that $q'_i$ is not a branching point. We consider the circle 
        $p(C_i)$ and
        ask $s_i$ to meet it transversely to it at only $1$ point. See 
        \Cref{fig:gen1}.
        We consider the graph $\G'$ as the union of the previous segments 
        and the $s_i$'s. Clearly the quotient surface retracts to it. Since it 
        contains
        all branching points its preimage $\hat{\G}$ is a spine for
        $\hat{\Si}$. It has 
        no univalent vertices since the $q_i$'s are branching points of a 
        Galois cover.
        
        \begin{figure}[!ht] 
            \centering
            \includegraphics[scale=0.65]{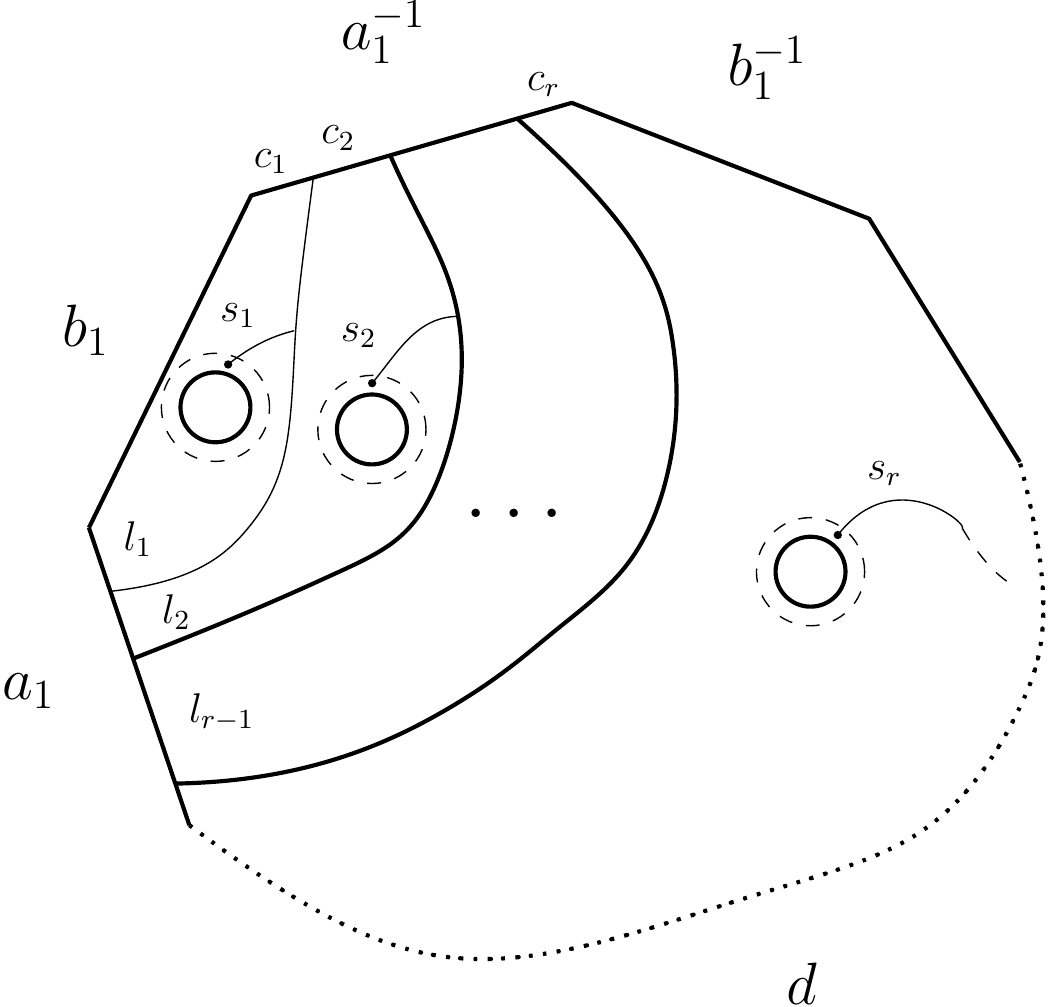}
            \includegraphics[scale=0.65]{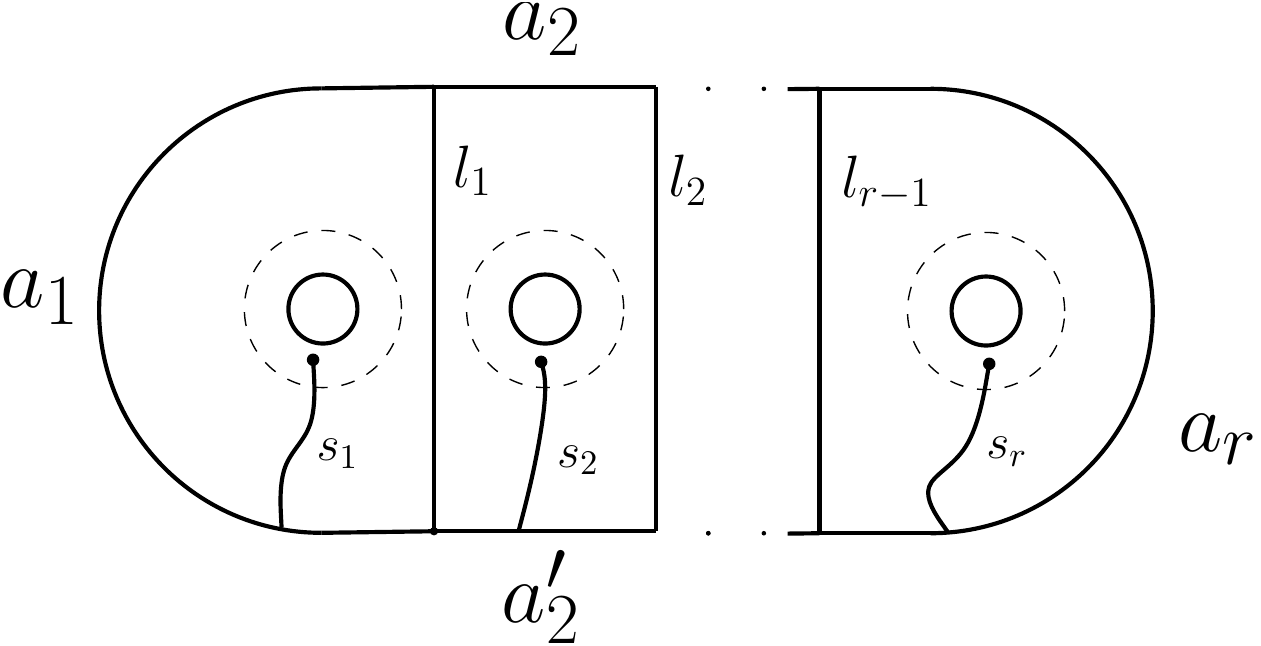}
            \caption{Drawing of $\G'$ for the case
                $genus(\hat{\Si}^{\hat{\phi}})\geq 1$ in the firts 
                image and
                $genus(\hat{\Si}^{\hat{\phi}})=0$ in the second.}
            \label{fig:gen2}
        \end{figure}
        
        \begin{figure}[!ht] 
            \centering
            \includegraphics[scale=0.7]{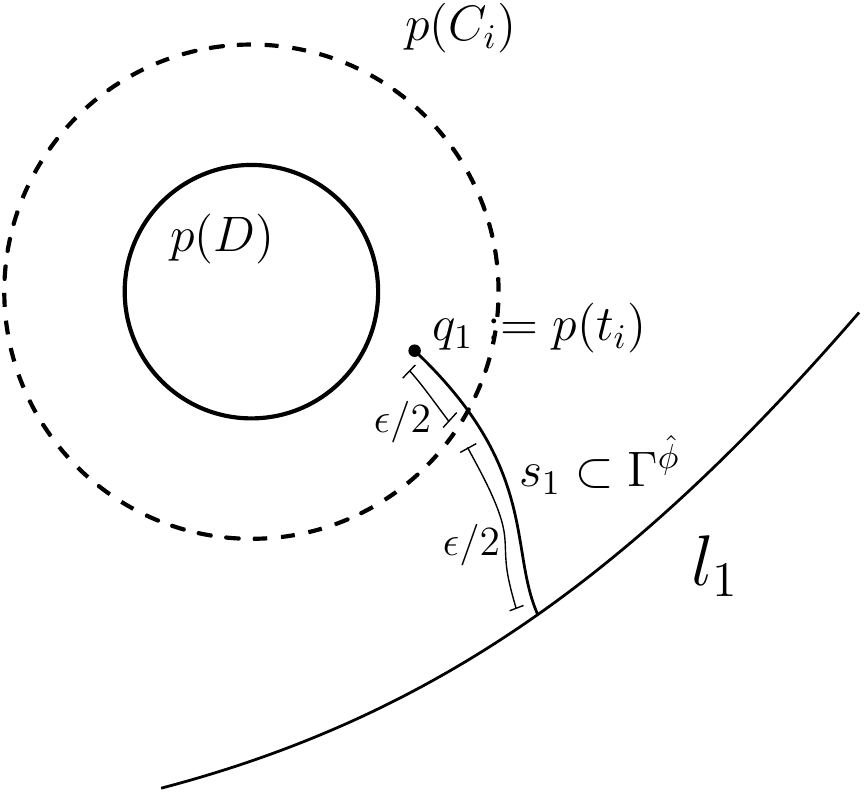}
            \caption{Neighbourhood of $q_1=p(t_i)$ in 
            $\hat{\Si}^{\hat{\phi}}$
                and edge $s_1$ joining $q_1$ and $l_1$.}
            \label{fig:gen1}
        \end{figure}
        
        In order to give a metric in the graph we proceed as follows. We give 
        the segments $d$ and $C_i$'s the same length they had in the proof of
        \cite[Theorem 5.12]{jmp}. We  impose every $s_i$ to have length some 
        small enough $\epsilon$ and the part of $s_i$ inside the cylinder 
        $\calC_i$ to have length $\epsilon/2$ (see \Cref{fig:gen1}). We give 
        each segment $l_i$ length $L-2\epsilon$.  It is easy to check that the 
        preimage graph $\hat{\G}$ with the pullback metric is \tat. 
        
        Now we consider the graph $\G:=\hat{\G}\cap \Si$ with the 
        restriction
        metric, except on the edges meeting $\partial\Si$ whose length is 
        redefined to be $\epsilon$. Along the lines of the proof of 
        \cite[Theorem 5.12]{jmp} we get that 
        $\phi_{(\G, \mathcal{P}, \sigma)}$ and $\phi$ are isotopic and 
        conjugate. If we denote by $\mathcal{P}$ the set of univalent vertices 
        of $\G$, it is an immediate consequence of the construction that 
        $(\G,\mathcal{P},\sigma)$ with the obvious permutation $\sigma$ of 
        $\mathcal{P}$ is a general \tat graph with $\phi_{(\G, \mathcal{P},
        \sigma)}|_{\G}=\phi|_{\G}$.
    \end{proof}
    
    \begin{example}
        We show an example that illustrates these ideas. Let $\Si$ be surface 
        of genus
        $1$ and $3$ boundary components $C_0,C_1,C_2$ embedded in $\R^3$ as in 
        the
        picture \ref{fig:examplestar}. Let 
        $\phi:\Si \rightarrow \Si$ be the restriction of the space rotation of 
        order $3$
        that exchanges the $3$  boundary components. We observe that in 
        particular
        $\phi^{3}|_{C_i}=id$ for $i=0,1,2$.
        
        We consider the star-shaped piece $S$ with $3$ arms together with the 
        order $3$
        rotation $r$ that exchanges the arms (see the picture 
        \Cref{fig:examplestar}).
        
        \begin{figure}[!ht]
            \centering
            \includegraphics[scale=0.6]{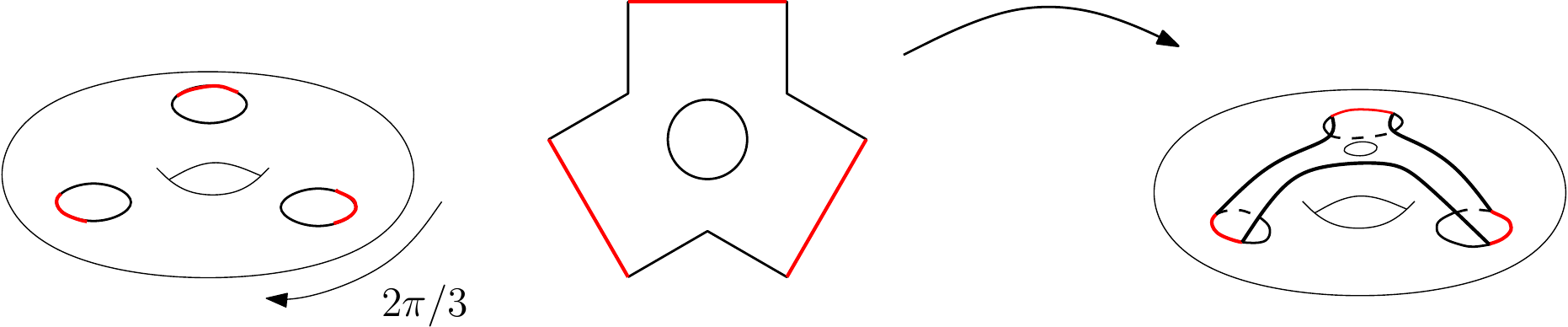}
            \caption{On the left, the torus $\Si$ with $3$ disks removed and 
            the orbit of an
                arc marked, that is, $3$ arcs in red. In the center, the 
                star-shaped piece $S$
                with $3$ arms to be glued to 
                the torus along those arcs. On the right, the surface we get 
                after gluing, with
                $2$ boundary components, one of them 
                invariant by the induced homeomorphism. 
            }
            \label{fig:examplestar}
        \end{figure}
        
        We glue $S$ to $\Si$ as the theorem indicates: we mark a small arc 
        $\alpha^0
        \subset C_2$ and all its iterated images by the rotation. 
        Then we glue $\alpha^0, \alpha^1, \alpha^2$ to $a^0,a^1,a^2$ 
        respectively by
        orientation reversing homeomorphisms. We get a new surface 
        $\hat{\Si}:=\Si \cup
        S$  
        with $2$ boundary components. We cap the boundary component that 
        intersects $C_0
        \cup C_1 \cup C_2$ with a disk $D^2$ and extend the homeomorphism to the
        interior of 
        the disk getting a new surface $\hat{\Si}$ and a homeomorphism
        $\hat{\phi}$. 
        
        Using Hurwitz formula $2-2g-4=0-2$ we get that the surface we are 
        gluing to
        $\Si$ has genus $0$ and hence it is a sphere with $4$ boundary 
        components. 
        See picture \Cref{fig:examplestar2}. Three of them are identified with 
        $C_0,
        C_1, C_2$, and the 4-th is called $C$ and is the only boundary 
        component of 
        $\hat{\Si}$.
        
        \begin{figure}[!ht]
            \centering
            \includegraphics[scale=0.5]{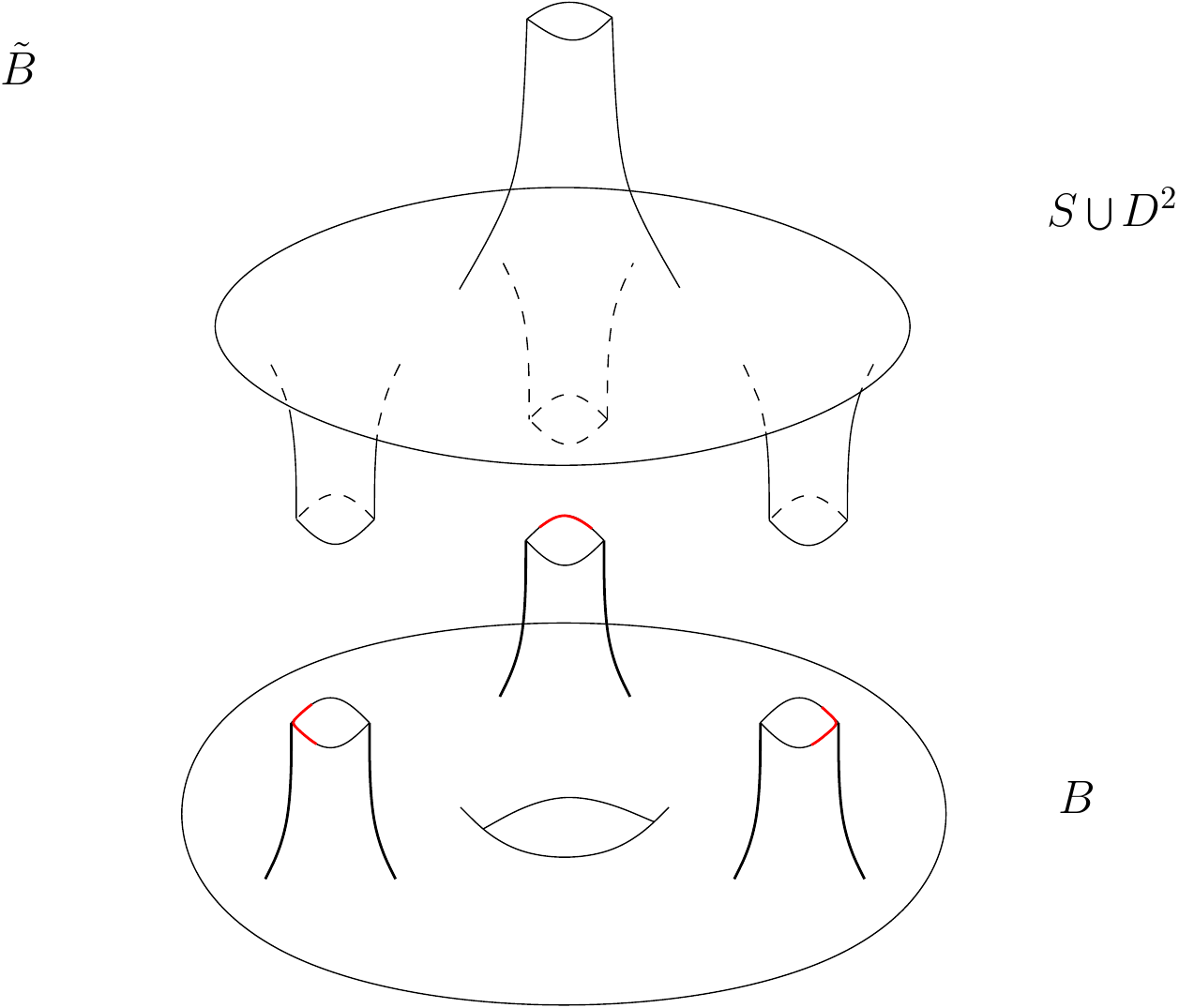}
            \caption{On the left, the torus with $3$ disks removed and $3$ the 
            orbit of an
                arc marked. On the right, the star-shaped piece with $3$ arms 
                to be glued to the
                torus along 
                those arcs.}
            \label{fig:examplestar2}
        \end{figure}
        
        We compute the orbit space $\hat{\Si}^{\hat{\phi}}$ by the 
        extended
        homeomorphism $\hat{\phi}$ and get a torus with $1$ boundary 
        component. 
        We consider the graph  $\G'$ as in picture~\Cref{fig:examplestar3}. We 
        put a
        metric in this graph. We set every edge of the hexagon to be $\pi/6 
        -\epsilon/3$
        long and the path joining 
        the hexagon with the branch point to be $\epsilon$ long. In this way, 
        if we look
        at the result of cutting $\hat{\Si}^{\hat{\phi}}$ along the 
        graph 
        $\hat{\phi}$ we see that the only boundary component that maps to 
        the
        graph by the gluing map has length $6 (\pi/6 -\epsilon/3)+2 \epsilon = 
        \pi$. 
        
        \begin{figure}[!ht]
            \centering
            \includegraphics[scale=0.5]{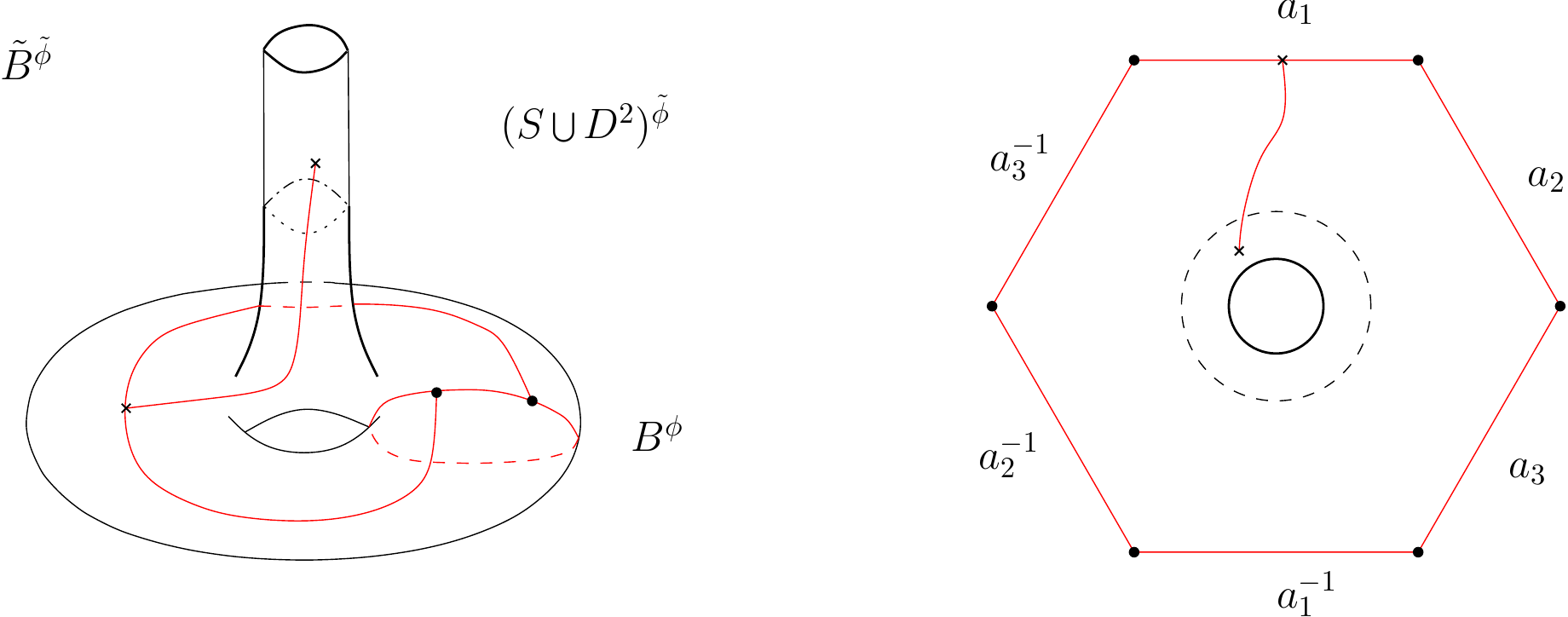}
            \caption{On the lower part we have the original surface. On the 
            upper part we
                have the surface that we attach, in this case a sphere with $4$ 
                holes removed.}
            \label{fig:examplestar3}
        \end{figure}
        
        The preimage $\hat{\G}$ of $\G'$ by the quotient map is a \tat 
        graph whose
        thickening is $\hat\hat{\Si}$. Its associated homeomorphism 
        $\hat{\phi}$ 
        leaves $\Si$ invariant and its restriction to it coincides with the 
        rotation
        $\phi$. Moreover 
        $(\hat{\G}\cap\Si,\hat{\G}\cap\partial\Si)$ is a
        general spine of 
        $(\Si,\partial\Si)$. Modifying the induced metric in 
        $\hat{\G}\cap\Si$ as
        in the proof of the Theorem and adding the order $3$ cyclic permutation 
        to the
        valency 
        $1$ vertices we obtain a \tat graph whose associated homeomorphism 
        equals
        $\phi$.
        
        \begin{figure}[H]
            \centering
            \includegraphics[scale=0.4]{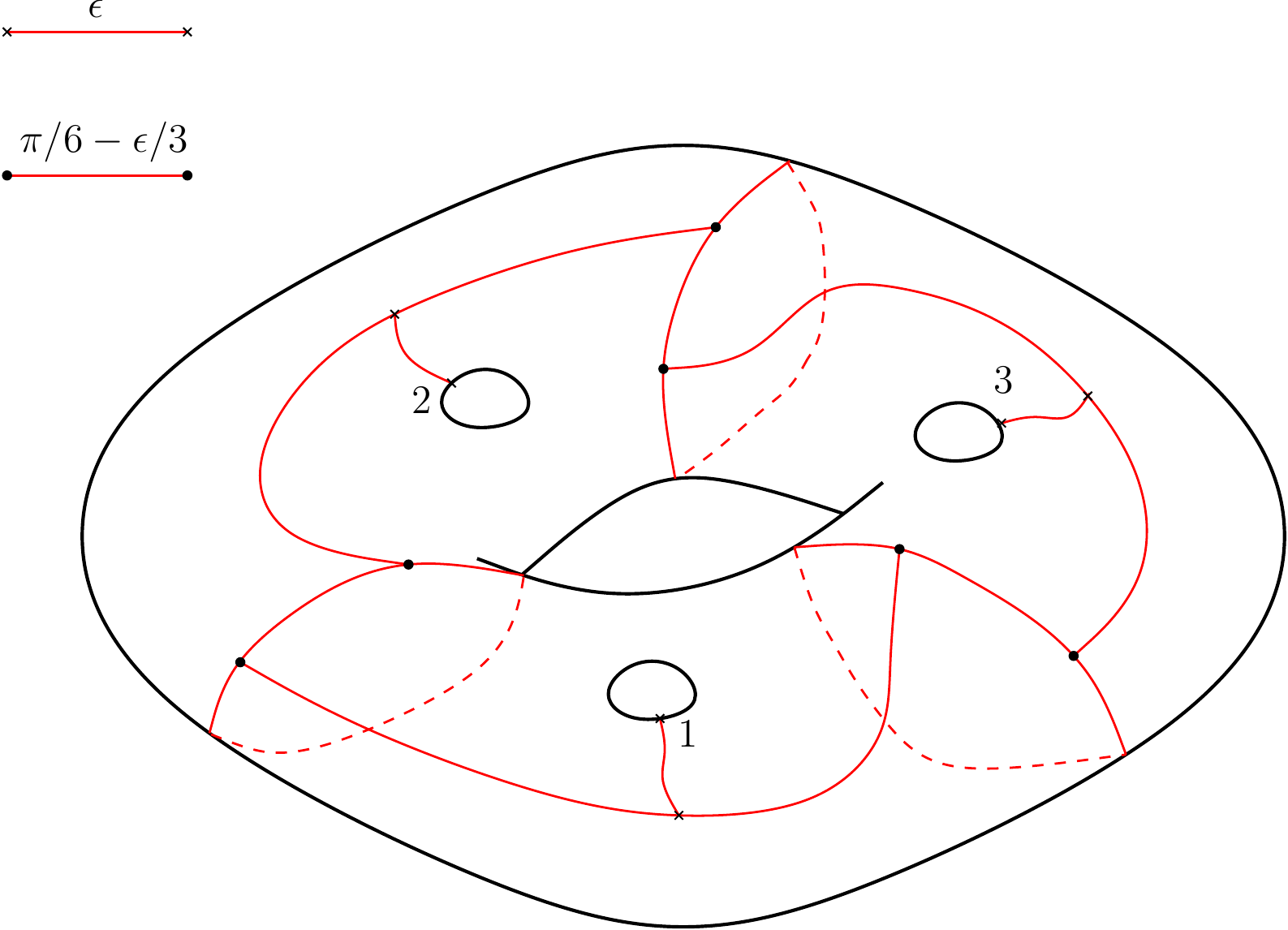}
            \caption{On the left, the torus with $3$ disks removed and $3$ the 
            orbit of an
                arc marked. On the right, the star-shaped piece with $3$ arms 
                to be glued to the
                torus along
                those arcs.}
            \label{fig:examplestar4}
        \end{figure}
    \end{example}
    
    \section{Seifert manifolds and plumbing graphs}\label{sec:seifert_man}
    
    In this subsection we recall some theory about Seifert manifolds and 
    plumbing graphs and fix the conventions used in this work. For more on this 
    topic, see \cite{Neum}, 
    \cite{Neum2}, \cite{Neum3}, \cite{Neum5}, \cite{Hir}, \cite{Hatch} or 
    \cite{Peder}. In many aspects we follow \cite{Peder}.
    
    \subsection{Seifert manifolds}
    
    Let $p,q\in \Z$ with $q >0$ and $\gcd(p,q)=1$. Let $D^2 \times [0,1]$ be a 
    solid
    cylinder. We consider the natural orientation on $D^2 \times [0,1]$.
    
    Let $(t, \theta)$ be polar coordinates for $D^2$. Let $r_{p/q}: D^2 
    \rightarrow
    D^2$ be the rotation $(t, \theta) \mapsto (t, \theta + 2\pi p/q)$. Let 
    $T_{p,q}$
    be the mapping torus of $D^2$ induced by the rotation $r_{p/q}$, that is, 
    the
    quotient space $$\frac{D^2 \times [0,1]} {(t,\theta,1)\sim (t,
        r_{p/q}(\theta),0)}.$$ If $p,p' \in \Z$ with $p \equiv p' \mod q$, then 
        the
    rotations $r_{p/q}$ and $r_{p'/q}$ are exactly the same map so
    $T_{p',q}=T_{p,q}$. The resulting space is diffeomorphic to a solid torus
    naturally foliated by circles which we call {\em fibers}. We call this 
    space a
    solid $(p,q)$-torus or a solid torus of type $(p,q)$. It has an orientation
    induced from the orientation of $D^2 \times [0,1] \subset \R^3$. The torus
    $\partial T_{p,q}$ is oriented as boundary of $T_{p,q}$.
    
    We call the image of $\{ (0,0) \} \times [0,1] \subset D^2 \times [0,1]$ in
    $T_{p,q}$ \textit{the central fiber} . We say that $q$ is the {\em 
    multiplicity}
    of the central fiber. If $q=1$ we call the central fiber a {\em typical 
    fiber}
    and if $q>1$ we call the central fiber a {\em special fiber}. Also any other
    fiber than the central fiber is called a typical fiber.
    
    If $a$ and $b$ are two closed curves in $\partial T_{p,q}$, let $[a] 
    \cdot[b]$
    denote the oriented intersection number of their classes in $H_1(\partial 
    T_{p,q}; \Z)$. 
    We
    describe $4$ classes of simple closed curves on $H_{1}(\partial T_{p,q}, 
    \Z)$:
    
    \begin{enumerate}
        \item A {\em meridian} curve $m:= \partial D^2 \times \{y\}$. We orient 
        it as boundary of $D^2 \times \{y\}$.
        
        \item A {\em fiber} $f$ on the boundary $\partial T_{p,q}$. We orient 
        it so that
        the radial projection on the central fiber is orientation preserving. It
        satisfies that $[m] \cdot [f] = q$.
        
        \item A {\em longitude} $l$ is a curve such that $[l]$ is a generator
        of  $H_1(T_{p,q}; \Z)$ and $[m] \cdot [l] 
        = 1$.
        
        \item A {\em section} $s$. That is a closed curve that intersects each 
        fiber
        exactly once. It is well defined up integral multiples of $f$. It is 
        oriented so
        that $[s] \cdot [f] = -1$. 
    \end{enumerate}

    \begin{figure}[ht]
        \centering
        \includegraphics[scale=0.3]{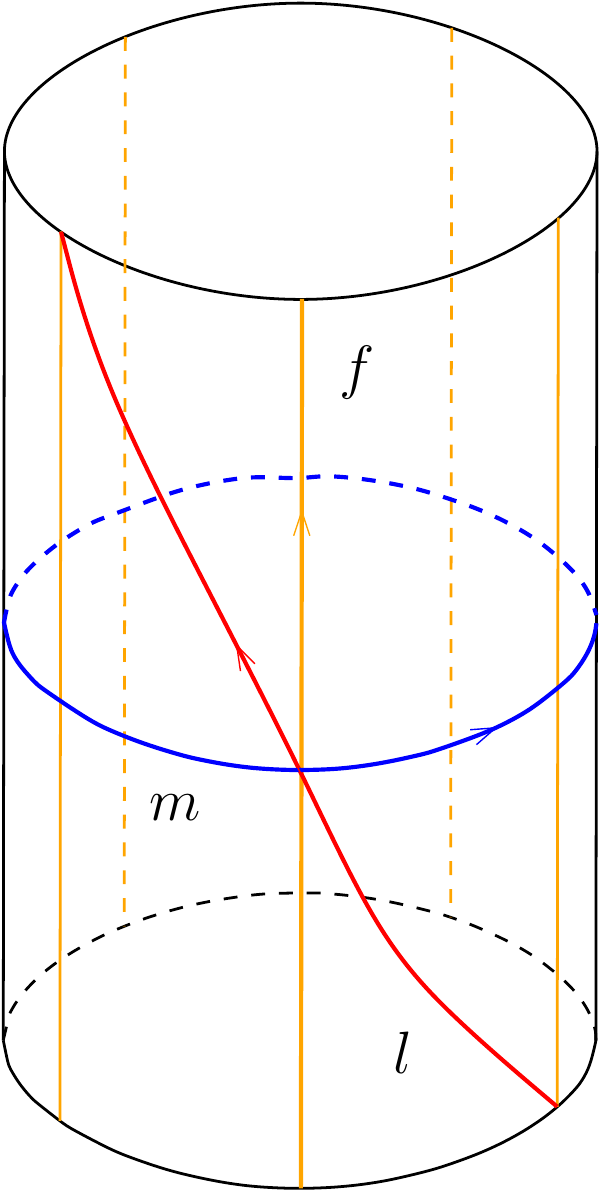}
        \caption{A torus $T_{2,5}$ with some closed curves marked on 
        its boundary. In orange a fiber $f$, in blue a meridian $m$ and in red 
        a longitude $l$.}
        \label{fig:pqtorus}
    \end{figure}

    We have defined two basis of the homology of $\partial T_{p,q}$, so we have 
    that
    there must exist unique $a, b \in \Z$ such that the equation
    \begin{equation}\label{eq:pqtorus}
    ([s] [f]) = ([m] [l]) \left( \begin{matrix} a&p\\ b& q \end{matrix} \right)
    \end{equation}
    holds in $H_1(\partial T_{p,q}; \Z)$. The matrix is nothing but a change of 
    basis. 
    
    The matrix of the equation has
    determinant $-1$ because $[s],[f]$ is a negative basis in the homology 
    group  
    $H_1(\partial T_{p,q}; \Z)$ . Therefore $bp \equiv 1 \mod q$. Changing the 
    class $[s]$ by adding integer multiples of $[f]$ to it, changes $b$ by 
    integer multiples of
    $q$. 
    
    We now fix conventions on Seifert manifolds. Let $B'$ be an oriented 
    surface of genus 
    $g$ and $r+k$ boundary components, $M':=B' \times \MS^1$  and  $s':M' 
    \times \MS^1 \rightarrow 
    B'$ the projection onto $B'$.   Let  $(\alpha_1, \beta_1), 
    \ldots, (\alpha_k, \beta_k)$ be $k$ pairs of integers with $\alpha_i > 0$ 
    for all $i=1, \ldots, k$.  Let $ N_1, \ldots N_k$ 
    be $k$ boundary tori on $M'$. On each of them consider the following two 
    curves $s_i := B' \times \{0\} \cap N_i$ and any fiber $f_i \subset N_i$. 
    Orient them so that $\{[s_i], [f_i]\}$ is a positive basis of $N_i$ as 
    boundary 
    of $M'$.  For each $i$, consider a solid torus $T_i = D^2 \times 
    \MS^1$ and the curves $m= \partial D^2 \times 
    \{0\}$ and $l:= \{pt\} \times \MS^1$ oriented so that $\{[m],[l]\}$ is a 
    positive basis of $T_i$. Attach $T_i$ to $N_i$ along its boundary by
     $$\left(
    \begin{smallmatrix} -\alpha_i & c \\ -\beta_i & d \end{smallmatrix} 
    \right):\partial T_i \rightarrow  N_i$$
    with respect to the two given basis. The numbers $c$ and $d$ are integers 
    such that the matrix has determinant $-1$. Note that, since the first 
    column defines the attaching of the meridian, the gluing is well defined up 
    to isotopy.
    
    The foliation on $N_i$ extends to all $T_i$ and gives it a structure of a 
    fibered solid torus. After gluing and extending the foliation to all $k$ 
    tori, we get a manifold $M$ and a collapsing map $s:M \to B$ where $B$ is 
    the surface of genus $g$ and $r$ boundary components.
    
    If a manifold $M$ can be constructed like this, we say that it is a {\em 
    Seifert  manifold} and the map $s:M \to B$ is a {\em Seifert fibering} for 
    $M$. We denote the resulting manifold after gluing $k$ tori by $$M(g,r, 
    (\alpha_1, \beta_1), \ldots, (\alpha_k, \beta_k)).$$
    
    Each pair $(\alpha_i, \beta_i)$ is called {\em Seifert pair} and we say 
    that it is normalized when $0 \leq \beta_i < \alpha_i$.

    We have, by definition and the discussion above, the following lemma and
    corollary.
    
    \begin{lemma}\label{lem:seifpq}
        Let $M \to B$ be a Seifert fibering. If a fiber $f$ has a neighborhood 
        diffeomorphic to a $(p,q)$-solid torus, then the there exists $b\in \Z$ 
        such that the (possibly unnormalized) Seifert invariant corresponding 
        to $f$ is $(q,-b)$ with $bp \equiv 1 \mod q$. Conversely, the special 
        fiber $f$ corresponding  to a Seifert pair $(\alpha, \beta)$ has a 
        neighborhood diffeomorphic as a  circle bundle to a 
        $(-c,\alpha_i)$-solid torus with $c \beta \equiv 1 \mod  \alpha$.
    \end{lemma}
    
    \begin{corollary}\label{cor:rotation_seifertinv}
        Let $\phi: \Si \rightarrow \Si$ be an orientation preserving periodic
        automorphism of a surface $\Si$ of order $n$ and let $\Si_\phi$ be the
        corresponding mapping torus. Let $x\in \Si$ be a point whose isotropy 
        group in the group $<\phi>$ has order $k$ with $n=k \cdot s$. Then 
        $\phi^{s}$ acts as a rotation in a disk around $x$ with rotation number 
        $p/k$ for some $p  \in \Z$ and the (possibly unnormalized) Seifert pair 
        of $M_\phi$ corresponding to the fiber passing through $x$ is $(k,-b)$ 
        with $bp \equiv 1 \mod k$.
    \end{corollary}
    
    \begin{proof}
        That $\phi^{s}$ acts as a rotation in a disk $D \subset \Si$ around $x$ 
        with
        rotation number $p/k$ for some $p \in \Z_{>0}$ follows from the fact 
        that $x$ is
        a fixed point for $\phi^{n/k}$. By construction of the mapping torus of 
        $\Si$ we
        observe that  the two mapping tori $M_{\phi|_D} \simeq  D_{\phi^{n/k}}$ 
        are
        diffeomorphic where $D$ is a small disk around $x$. By definition of 
        fibered
        torus we have that $ D_{\phi^{n/k}} \simeq T_{p,k}$. The rest follows 
        from
        \Cref{lem:seifpq} above.
    \end{proof}

    \subsection{Plumbing graphs}
    
    A plumbing graph is a decorated graph
    that encodes the information to recover the topology of a certain 
    $3$-manifold. As with Seifert manifolds, we fix notation and conventions.
    
    This is the decoration and its corresponding meaning:
    
    \begin{itemize}
        \item Each vertex corresponds to a circle bundle. It is decorated with 
        $2$ integers $e_i$ (placed on top) and $g_i$ placed on bottom. If a 
        vertex has valency $v_i$ consider the circle bundle over the surface of 
        genus $g_i$ and $v_i$ boundary components and pick a section on the 
        boundary so that the global Euler number is $e_i$. When $g$ is omitted 
        it is assymed to be $0$.
        \item Each edge  tells us that the circle bundles corresponding to the 
        ends of the edge are glued along a boundary torus by the gluing map 
        $J(x,y)=(y,x)$ defined with respect to $\mathrm{section} \times 
        \mathrm{fiber}$ on each boundary torus.
        \item An and ending in an arrowhead  represents that an open solid 
        torus is  removed from the corresponding circle bundle from where the 
        edge comes out.
    \end{itemize}
    
    The construction of the $3$-manifold associated to a plumbing graph is clear
    from the description of its decoration above.
    
    We point out a minor correction to an argument in \cite{Neum} and reprove 
    a  known lemma  which is crucial in \Cref{sec:algor} (see discussion 
    afterwards in \Cref{rem:fix_neum}).
    
    \begin{lemma}\label{lem:bamboo}
        Let $\Lambda$ be a plumbing graph.
        \begin{itemize}
            \item[1)] If a portion of $\Lambda$ has the following form:
            \begin{figure}[H]
                \centering
                \includegraphics[scale=0.6]{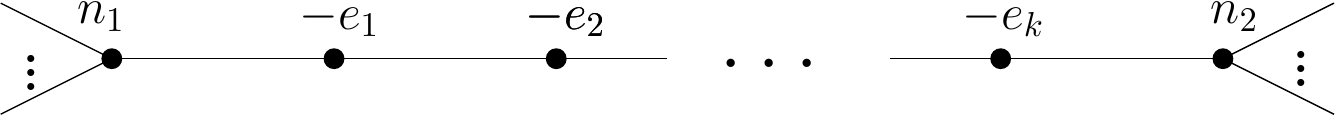}
                \label{fig:chain}
            \end{figure}
            Then the  piece corresponding to the  node $n_1$ is glued to the 
            piece
            corresponding to the node $n_2$ along a torus by the matrix 
            $G=\left(
            \begin{smallmatrix} a&b\\ c& d \end{smallmatrix} \right)$ with 
            $\det (G) =-1$
            and where $-b/a=[e_1,\ldots, e_k]$ with the numbers in brackets 
            being the
            continued fraction $$\frac{-b}{a}= e_{1} - \frac{1}{e_{2}- 
            \frac{1}{e_3-
                    \frac{1}{\ldots}}}.$$
            
            \item[2)] If a portion of $\Lambda$ has the following form:
            
            \begin{figure}[H]
                \centering
                \includegraphics[scale=0.6]{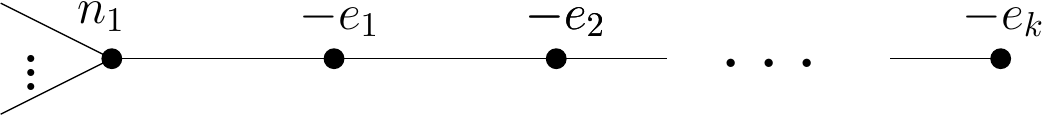}
                \label{fig:chain_1}
            \end{figure}
        \end{itemize}
       Then the piece corresponding to the node $n_i$ is glued along a torus to 
       the
   boundary of a solid torus $D^2 \times \MS^1$ by the matrix $\left(
   \begin{smallmatrix} a&b\\ c& d \end{smallmatrix} \right) $ with $-d/c=[e_1, 
   e_2,
   \ldots, e_k]$. 
\end{lemma}

    \begin{proof}
        Let $T:=D^2 \times \MS^1$ be a solid torus naturally foliated by 
        circles 
        by its
        product structure. Let $s$ be the closed curve $\partial D^2 \times 
        \{0\}$ and
        let $f$ be any fiber on the boundary of the solid torus. Orient them so 
        that
        $\{[s],[f]\}$ is a positive basis of $H_1(\partial T; \Z)$. If $T$, 
        $T'$ are two
        copies of the solid torus. Then $M_i$ is the $S^1$-bundle $ T 
        \sqcup_{E_i} T'$
        where $E_i: \partial T \rightarrow \partial T'$ is the matrix $$\left(
        \begin{matrix} -1&0\\ e_i& 1 \end{matrix} \right)$$ used in the gluing 
        along the
        boundaries. In particular $[f] = [f']$ in $H_1(M_i; \Z)$. The $-1$ in 
        the upper
        left part reflects the fact that $s$ inherits different orientations 
        from the
        two tori.

        We treat the case $1)$ first. If $M_i$ is the piece corresponding to 
        the node
        $n_i$ with $i=1,2$ we have that the gluing from $M_1$ to $M_2$ is
        
        $$M_1 \sqcup _ J (A \times \MS^1 \sqcup_{E_1} A \times \MS^1) 
        \sqcup_{J} 
        \cdots
        \sqcup _ J (A \times \MS^1 \sqcup_{E_k} A \times \MS^1) \sqcup_{J} M_2 
        $$
        
        Where $A \times \MS^1$ is the trivial circle bundle over the annlus
        $A:=[1/2,1] \times \MS^1$.  Let $(r,\theta)$ be polar 
        coordinates for $A$.  The  two tori forming the  boundary of  $A \times 
        \MS^1$ are oriented as boundaries 
        of $A \times \MS^1$. Observe that the map  $r((1/2, \theta) , \eta) = 
        ((1, \theta) ,\eta)$ is orientation reversing.
        
        We define $s= \{S_{1/2}^1\} \times \{0\}$ and $f= \{(1/2,0\} \times 
        \MS^1$ and
        orient them so that the ordered basis $\{[s],[f]\}$ is a  positive 
        basis for
        $H_1(S_{1/2}^1 \times \MS^1; \Z)$. We define similarly $s'=\{S_{1}^1\} 
        \times
        \{0\}$, $f'= \{(1,0\} \times \MS^1$ and orient them so that 
        $\{[s'].[f']\}$ is a
        positive basis for $H_1(S_{1}^1 \times \MS^1; \Z)$. Then the homology 
        classes
        $[r(s)]$ and $[r(f)]$ form a negative basis. In fact $[s']=-[r(s)]$ and
        $[f]=[r(f)]$. This is the reason of the matrices $\left( 
        \begin{smallmatrix}
        -1&0\\ 0&1 \end{smallmatrix} \right)$ in the \Cref{eq:first} below.
        
        So the  gluing matrix $G$ from a torus in the boundary of $M_1$ to a 
        torus in
        the boundary of $M_2$ is given by the following composition of matrices:
        \begin{equation}\label{eq:first}
        \begin{split}
        G &= \left( \begin{smallmatrix} 0&1\\ 1&0 \end{smallmatrix} \right)
        \left( \begin{smallmatrix} -1&0\\ 0&1 \end{smallmatrix} \right) 
        \left( \begin{smallmatrix} -1&0\\ e_k&1 \end{smallmatrix} \right)
        \left( \begin{smallmatrix} -1&0\\ 0&1 \end{smallmatrix} \right)
        \left( \begin{smallmatrix} 0&1\\ 1&0 \end{smallmatrix} \right)\cdots  
        \left( \begin{smallmatrix} 0&1\\ 1&0 \end{smallmatrix} \right)
        \left( \begin{smallmatrix} -1&0\\ 0&1 \end{smallmatrix} \right) 
        \left( \begin{smallmatrix} -1&0\\ e_1&1 \end{smallmatrix} \right)
        \left( \begin{smallmatrix} -1&0\\ 0&1 \end{smallmatrix} \right)
        \left( \begin{smallmatrix} 0&1\\ 1&0 \end{smallmatrix} \right) \\
        &=\left( \begin{smallmatrix} 0&1\\ 1&0 \end{smallmatrix} \right) 
        \left( \begin{smallmatrix} -1&0\\ -e_k&1 \end{smallmatrix} \right)
        \left( \begin{smallmatrix} 0&1\\ 1&0 \end{smallmatrix} \right)\cdots  
        \left( \begin{smallmatrix} 0&1\\ 1&0 \end{smallmatrix} \right) 
        \left( \begin{smallmatrix} -1&0\\ -e_1&1 \end{smallmatrix} \right)
        \left( \begin{smallmatrix} 0&1\\ 1&0 \end{smallmatrix} \right) \\
        &=\left( \begin{smallmatrix} 0&1\\ 1&0 \end{smallmatrix} \right)
        \left( \begin{smallmatrix} 0&-1\\ 1&-e_k \end{smallmatrix} \right) 
        \cdots 
        \left( \begin{smallmatrix} 0&-1\\ 1&-e_1 \end{smallmatrix} \right)
        \end{split}
        \end{equation}
        
        Observe that each matrix in the definition of $G$ has determinant $-1$ 
        so $\det(G)=-1$ because there is an odd number. Hence $G$ inverts 
        orientation on the boundary tori, preserving the orientation on the 
        global $3$ manifoldd. The result about the continued fraction follows 
        easily by induction on $k$.
        
        Now we treat similarly the case $2)$. The gluing from a boundary torus 
        from $M_1$ to $\partial D^2 \times \MS^1$ is
        $$
        M_1 \sqcup _ J (A \times \MS^1 \sqcup_{E_1} A \times \MS^1) 
        \sqcup_{J} 
        \cdots
        \sqcup _ J (A \times \MS^1 \sqcup_{E_k} D^2 \times \MS^1). 
        $$
        
        Hence, by a similar argument to the previous case, the matrix that 
        defines the
        gluing is
        \begin{equation}
        \begin{split}
        G &= \left( \begin{smallmatrix} -1&0\\ e_k&1 \end{smallmatrix} 
        \right)\left(
        \begin{smallmatrix} -1&0\\ 0&1 \end{smallmatrix} \right)\left(
        \begin{smallmatrix} 0&1\\ 1&0 \end{smallmatrix} \right)\cdots  \left(
        \begin{smallmatrix} 0&1\\ 1&0 \end{smallmatrix} \right)\left(
        \begin{smallmatrix} -1&0\\ 0&1 \end{smallmatrix} \right) \left(
        \begin{smallmatrix} -1&0\\ e_1&1 \end{smallmatrix} \right)\left(
        \begin{smallmatrix} -1&0\\ 0&1 \end{smallmatrix} \right)\left(
        \begin{smallmatrix} 0&1\\ 1&0 \end{smallmatrix} \right) \\
        &= \left( \begin{smallmatrix} -1&0\\ e_k&1 \end{smallmatrix} \right) 
        \left(
        \begin{smallmatrix} -1&0\\ 0&1 \end{smallmatrix} \right) \left(
        \begin{smallmatrix} 0&1\\ 1&0 \end{smallmatrix} \right) \left(
        \begin{smallmatrix} 0&-1\\ 1&-e_{k-1} \end{smallmatrix} \right)  \cdots 
        \left(
        \begin{smallmatrix} 0&-1\\ 1&-e_1 \end{smallmatrix} \right) \\
        &= \left( \begin{smallmatrix} 0&1\\ 1&-e_k \end{smallmatrix} \right) 
        \left(
        \begin{smallmatrix} 0&-1\\ 1&-e_{k-1} \end{smallmatrix} \right)  \cdots 
        \left(
        \begin{smallmatrix} 0&-1\\ 1&-e_1 \end{smallmatrix} \right)
        \end{split}
        \end{equation}
        
        By the expression in the last line we see that all matrices involved 
        but the one
        on the left, have determinant $1$ so we get $det(G)=-1$. Again, by 
        induction on
        $k$ the result on the continued fraction follows straight from the las 
        line.
    \end{proof}
    
    \begin{remark}\label{rem:fix_neum}
        Note the differences of the  Lemma above with Lemma 5.2 and the 
        discussion before it in  \cite{Neum}: 
        there the author  does not observe that in each piece $A 
        \times \MS^1$, the natural projection from one boundary torus to the 
        other is orientation reversing. So the matrices $\left( 
        \begin{smallmatrix} -1&0\\ 0&1 \end{smallmatrix} \right)$ are not taken 
        into account there. 
        
         In a more extended manner. The problem is with the claim that the 
         matrix $C$ (in equation $(\ast)$ 
         pg.$319$ of \cite{Neum}) is the gluing matrix.  The equation above 
         equation $(\ast)$ in that page, describes the gluing between the two 
         boundary tori as 
         a concatenated gluing of several pieces. In particular you glue a 
         piece of the form  $A \times S^1$ with another piece of the same form 
         using the matrix $H_k$ and then you glue these pieces a long $J$-
         matrices. 
         Then it is claimed that ``since $A \times S^1$ is a collar'' then the 
         gluing matrix (up to a sign) is $J H_k J \cdots J H_1 J$.  However, 
         notice that each piece $A \times S^1$ has two boundary tori, 
        and they inherit "opposite" orientations. More concretely, the natural 
        radial projection from one boundary torus to the other is orientation 
        reversing. So even, if they are a collar (which they are), they 
        interfere somehow in the gluing. That is why we add the matrices 
        $\left( 
        \begin{smallmatrix} -1&0\\ 0&1 \end{smallmatrix} \right)$ between each 
        $J$ and each $H_k$ matrix.
    \end{remark}

    \section{Horizontal surfaces in Seifert manifolds}\label{sec:hor_sur}
    
    In this section, we study and classify horizontal surfaces of Seifert 
    fiberings up to isotopy. The results contained here are known. The 
    exposition that we choose to do here is 
    useful for \Cref{sec:algor}.
    
    We recall that we are only considering Seifert manifolds that are orientable
    with orientable base space and with non-empty boundary. Horizontal surfaces 
    in a orientable Seifert manifold with orientable base 
    space are always orientable (see for example Lemma 3.1 in \cite{Zull}). So 
    by our assumptions only orientable horizontal surfaces appear. Let $F$ be 
    any  fiber of the Seifert fibering $M \to B$. 
    
    \begin{definition}
        Let $H$ be a surface with non-empty boundary which is properly embedded 
        in $M$ i.e. $H \cap \partial M = \partial H$. We say that $H$ is a {
        \em horizontal surface} of $M$ if it is transverse to all the fibers of 
        $M$. 
    \end{definition}

    \begin{definition}\label{def:iso_equiv}
        Let $H(M)$ be the set of all horizontal surfaces of $M$, we define
        $$\mathcal{H}(M):= H(M) / \sim$$ where two elements $H_1, H_2 \in H(M)$ 
        are related $H_1 \sim H_2$ if their inclusion maps are isotopic.
    \end{definition}

    Let $n:= \lcm(\alpha_1, \ldots, \alpha_k)$.  We consider the action of the
    subgroup of the unitary complex numbers given by the $n$-th roots of unity
    $c_n:=\{e^{2\pi i m/n}\}$ with $m=0, \ldots, n-1$ on the fibers of $M$ . The
    element $e^{2\pi i m/n}$ acts on a typical fiber by a rotation of $2 \pi 
    m/n$
    radians and acts on a special fiber with multiplicity $\alpha_i$ by a 
    rotation
    of $2 \pi m \alpha_i/n$ radians.
    
    We quotient $M$ by the action of this group and denote $\hat{M}=M/ c_n$ the
    resulting quotient space. By definition, the action of $c_n$ preserves the
    fibers and is effective. The manifold $\hat{M}$ is then a Seifert manifold 
    where
    we have {\em killed} the  multiplicity of all the special fibers of $M$. 
    Hence
    it is a locally trivial $S^1$-fibration over $B$ and since $\partial  B \neq
    \emptyset$, it is actually a trivial fibration so $\hat{M}$ is 
    diffeomorphic to
    $B \times \MS^1$.
    
    Let $\pi:M \rightarrow \hat{M}$ be the quotient map induced by the action of
    $c_n$.  Observe that $\hat{M}$, seen as a Seifert fibering with no special
    fibers, has the same base space as $M$ because the action given by $c_n$
    preserves fibers. In particular we have the following commutative diagram
    
    \begin{equation}\label{diag:seif_orbit}
    \begin{tikzcd}
    M \arrow[rightarrow]{r}{\pi} \arrow{d}{s}&\hat{M} \arrow{dl}{\hat{s}}\\
    B
    \end{tikzcd}
    \end{equation}
    
    Where $s$ (resp. $\hat{s}$) is the projection map from the Seifert fibering 
    $M$
    (resp. $\hat{M}$) onto its base space $B$.

    \begin{definition}Let $H$ be a horizontal surface in $M$. We say that $H$ is
        well embedded if it is invariant by the action of $c_n$.
    \end{definition}
     A horizontal surface $H$ defines a linear map $H_1(M; \Z) \to \Z$ by 
     considering its Poincare dual. If $H$ intersects a generic fiber $m$ 
     times, then it intersects a special fiber with multiplicity $\alpha$, 
     $m/\alpha \in \Z$ times. This is because a generic fiber covers that 
     special fiber  $\alpha$ times. Hence, by isotoping any horizontal fiber, 
     we can always find well-embedded  representatives $\hat{H} \in [H]$.

    \begin{remark}\label{rem:equivariant_iso}
        Observe that if $H$ and and $H'$ are two well-embedded surfaces with 
        $[H]=[H']$,
        then we can always find a fiber-preserving isotopy $h$ that takes the 
        inclusion
        $i:H \hookrightarrow M$ to an homeomorphism $h(\cdot,1):H \rightarrow 
        H'$  such that $h(H,t)$ is a well-embedded surface for all $t$. This 
        fact help us prove the following:
    \end{remark}
    
    \begin{lemma}\label{lem:11_horizontal}
        There is a bijection $\pi_{\sharp}:\mathcal{H}(M) \rightarrow
        \mathcal{H}(\hat{M})$ induced by $\pi$.
    \end{lemma}
    
    \begin{proof}
        Let $[H] \in \mathcal{H}(M)$ and suppose that $H \in [H]$ is a 
        well-embedded
        representative. Then clearly $\pi(H) \in H(\hat{M})$. If $H'$ is another
        well-embedded representative of the same class, then by
        \Cref{rem:equivariant_iso} we have that $[\pi(H)] = [\pi(H')]$ in
        $\mathcal{H}(\hat{M})$. Hence the map $\pi_{\sharp}([H]):= [\pi(H)]$ is 
        well
        defined.
        
        The map $\pi_{\sharp}$ is clearly surjective because 
        $\pi^{-1}(\hat{H})$ is a
        well-embedded surface for any horizontal surface $\hat{H} \in 
        H(\hat{M})$ and
        hence $\pi_{\sharp}([\pi^{-1}(\hat{H})])=[\hat{H}]$.

        Now we prove that the natural candidate for inverse 
        $\pi^{-1}_\sharp([\hat{H}]):=[\pi^{-1}(\hat{H})]$ is well-defined. Let
        $[\hat{H}] \in  \mathcal{H}(\hat{M})$ with $\hat{H} \in [\hat{H}]$ a
        representative of the class. Let $H:=\pi^{-1}(\hat H)$. If
        $[\hat{H}]=[\hat{H}']$ for some $\hat{H}'$ in $H(\hat{M})$ then
        $[\pi^{-1}(\hat{H}')]=[\pi^{-1}(\hat{H})]$ by just pulling back the 
        isotopy
        between $\hat{H}$ and $\hat{H}'$ to $M$ by the map $\pi$. Hence the map 
        is well
        defined. By construction, it is clear that for any $H \in 
        \mathcal{H}(M)$ we
        have that $\pi_\sharp^{-1}(\pi_\sharp([H])=[H]$ so we are done.
    \end{proof}

    The objective of this section is to study $\mathcal{H}(M)$ but because of
    \Cref{lem:11_horizontal} above, it suffices to study $\mathcal{H}(\hat M)$.
    
    Fix a trivialization  $\hat{M} \simeq B  \times \MS^1$ once and for all. We
    observe that since $\partial B \neq \emptyset$, the surface $B$ is 
    homotopically
    equivalent to a wedge of $\mu=2g+r-1$ circles, denote this wedge by 
    $\tilde{B}$.
    Observe that $\mathcal{H}(M)$ is in bijection with multisections of 
    $\tilde{B} \times \MS^1 \rightarrow \tilde{B}$ up to isotopy. Multisections 
    are multivalued continuous maps from $\tilde{B}$ to $\tilde{B} \times 
    \MS^1$.
    
    \begin{lemma}\label{lem:class_hor}
        The elements in $\mathcal{H}(\tilde B \times \MS^1)$  are in bijection 
        with elements of $$H^1(\tilde B \times \MS^1; \Z) = H^1(\tilde B; \Z) 
        \oplus \Z$$ that are not in $H^1(\tilde B; \Z) \oplus \{0\}$. Oriented 
        horizontal surfaces that intersect positively any fiber of $\tilde{B} 
        \times \MS^1$ are in bijection with elements of $H^1(\tilde B; \Z) 
        \oplus \Z_{>0}.$
    \end{lemma}
    
    \begin{proof}
        We have that $H^1(B; \Z) \oplus \Z = \Z^{\mu} \oplus \Z$. Given an 
        element
        $(p_1, \ldots, p_{\mu}, q) = k((p_1', \ldots, p_{\mu}', q'))  \in 
        \Z^{\mu}
        \times \Z$ with $q \neq 0$ and $(p_1', \ldots, p_{\mu}', q')$ 
        irreducible
        (seeing $\Z^{\mu} \times \Z$ as a $\Z$-module). Let
        $\frac{p_j'}{q'}=\frac{k_jp_j''}{k_jq''}$ with $p_j''/q''$ an 
        irreducible
        fraction. Consider in each $S_k^1 \times \MS^1$, $k_j$ disjoint copies 
        of 
        the
        closed curve of slope $p_j''/q''$. We denote the union of these $k_j$ 
        copies by
        $\tilde H_j$. We observe that $ \tilde H_j$ intersects $C$ in 
        $k_jq''=q'$ points
        for each $j$. We can, therefore, isotope the connected components of 
        each
        $\tilde H_j$ so that $\bigcup_j \tilde H_j$ intersects $C$ in just $q'$ 
        points.
        We do so and consider the set $\bigcup_j \tilde H_j$. The horizontal 
        surface
        $\tilde H$ of $\tilde{B}\times \MS^1$ associated to $k((p_1', \ldots, 
        p_{\mu}',
        q')$ is $k$ disjoint parallel copies of $\bigcup_j \tilde H_j$. 
        
        On the other direction, given an element $[H]\in \mathcal{H}(\tilde B 
        \times
        \MS^1)$. Let $[H]$ denote also the class of any horizontal surface in 
        $H_1(\tilde
        B \times \MS^1; \Z)$ and we simply have $q = [H]\cdot C$. And we have 
        $p_i = [H]
        \cdot [S_i^1 \times \{0\}]$. That is, the corresponding element in 
        $H^1(B; \Z) \oplus \Z$ is the Poincar\'e dual of the class of $H$ in 
        the homology of $\tilde{B} \times \MS^1$.
    \end{proof}
    
    \begin{lemma}
        $\tilde H$ is connected if and only if the element $(p_1, \ldots, 
        p_{\mu}, q)$
        is irreducible in $H^1(\hat{M}; \Z) \simeq  H^1(\tilde B; \Z) \oplus 
        \Z$. 
    \end{lemma}
    
    \begin{proof}
        We know that by construction $\tilde{H} \cap C$ are $q$ points. It is 
        enough to
        show that these $q$ points lie in the same connected component since 
        any other
        part of $\tilde{H}$ intersects some of these points. We label the points
        cyclically according to the orientation of $C$. So we have $c_1, 
        \ldots, c_q \in
        C$. We recall that  $S_j^1 \times \MS^1 \cap \tilde H$ is formed by 
        $k_j$  parallel copies of the closed curve of slope $p_j'/q'$ with 
        $\frac{k_jp_j'}{k_j q'}= \frac{p_j}{q}$. Hence the point $x_i$ is 
        connected by these curves with 
        the points $c_{i+ tk_j \mod q}$. Since $(p_1, \ldots, p_{\mu}, q)$ is 
        irreducible then $\gcd(p_1, \ldots, p_\mu,q)=1$ and hence $\gcd(k_1, 
        \ldots, 
        k_\mu)=1$. Therefore the equation
        
        $$i + t_1k_1 + \cdots + t_\mu k_\mu = j \mod q$$ admits an integer 
        solution on
        the variables $t_1, \ldots, t_\mu$ for any two $i,j \in \{1, \ldots, 
        q\}$. This
        proves that the points $c_i$ and $c_j$ are in the same connected 
        component in
        $\tilde{H}$.
        
        Conversely if the element is not irreducible, then $(p_1, \ldots, 
        p_{\mu}, q)= k
        (p_1', \ldots, p_{\mu}', q')$ for $(p_1', \ldots, p_{\mu}', q')$ 
        irreducible and
        $k>1$. Then, by construction, $\tilde H$ is formed by $k$ disjoint 
        copies of the
        connected horizontal surface associated to $(p_1', \ldots, p_{\mu}', 
        q')$ 
    \end{proof}

    \subsection*{Handy model of a Seifert fibering.}\label{subsec:handy_model}

    We describe a particularly handy model of the Seifert fibering that we use 
    in
    \Cref{sec:algor}. The idea is taken from a construction in \cite{Hatch}. 
    For each $i=1, \ldots, k$ let $x_i \in B$ be the image by 
    $s:M
    \rightarrow B$ of the special fiber $F_i$. We pick one boundary component 
    of the
    base space and denote it by $L$. For each $i=1, \ldots, k$ pick an arc $l_i$
    properly embedded in $B$ and with the end points in $L$ (i.e. with $l_i 
    \cap L =
    \partial l_i$) in such a way that cutting along $l_i$ cuts off a disk $D_i$ 
    that
    contains $x_i$ and no other point from $\{x_1, \ldots, x_k\}$. We pick a
    collection of such arcs $l_1, \ldots, l_k$ pairwise disjoint. We define 
    $$B':=B
    \setminus \bigsqcup_i int(D_i)$$ where $int(\cdot)$ denotes the interior.  
    See
    \Cref{fig:base_seif} below and observe that $B$ and $B'$ are diffeomorphic. 
    
    \begin{figure}[h]
        \centering
        \includegraphics[scale=0.5]{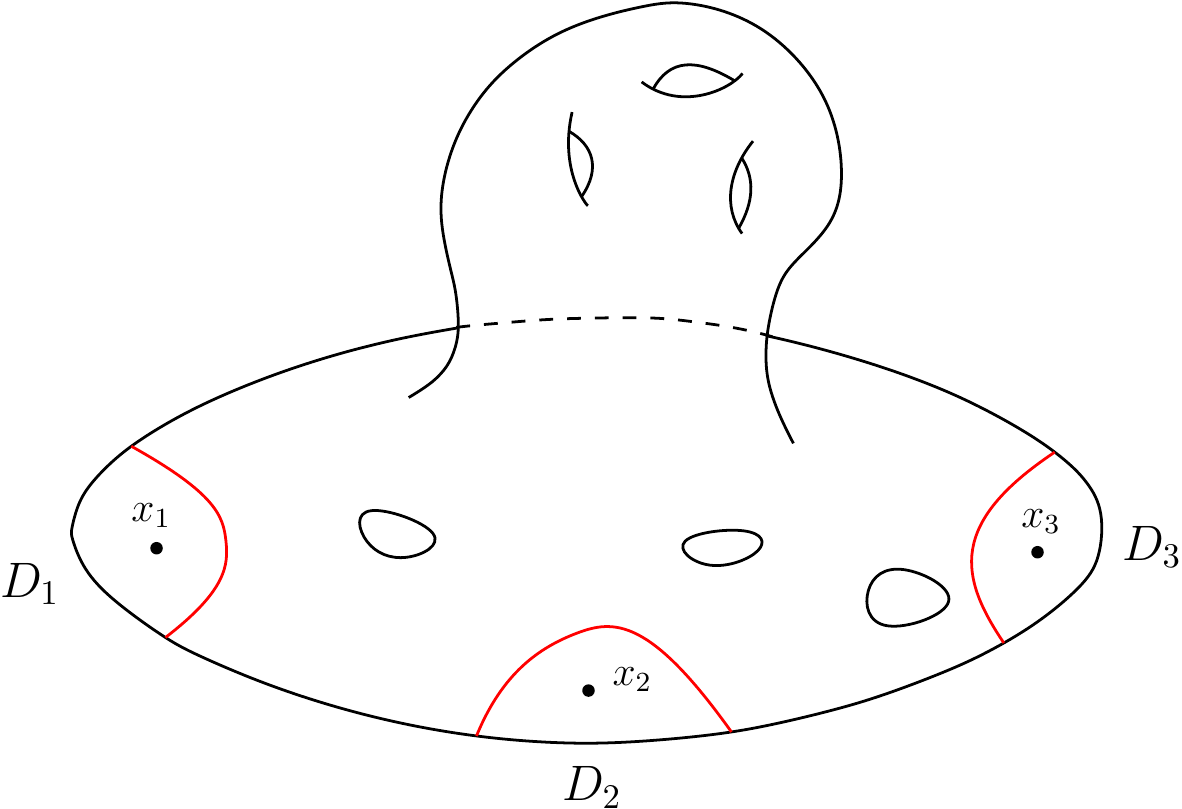}
        \caption{We see the base space $B$ of a Seifert manifold. It has genus 
        $3$ and
            $4$ boundary components. The $3$ points are the image of the 
            special fibers by
            the projection $s$ and if we cut along the three red arcs, we get 
            the surface
            $B'$.}
        \label{fig:base_seif}
    \end{figure}
    
    Let $M':=s^{-1}(B')$. Since $M'$ contains no special fibers and $\partial B'
    \neq \emptyset$ then $M'$ is diffeomorphic as a circle bundle to $B' \times
    \MS^1$. Recall that  $s^{-1}(D_i)$ is a solid torus of type 
    $(p_i,\alpha_i)$ 
    with
    $p_i \beta_i \equiv 1 \mod \alpha_i$ (see \Cref{lem:seifpq}). 
    
    Summarizing, the { \em handy model } consists of:
    \begin{enumerate}
        \item[i)] A system of arcs $l_1, \ldots, l_k$ as explained.
        \item[ii)] A trivialization of $M'$, that is an identification of $M'$ 
        with $B'
        \times \MS^1$. 
        \item[iii)] Identifications of $s^{-1}(D_i)$ with the corresponding 
        model
        $T_{p_i,\alpha_i}$ for each $i=1, \ldots, k$.
    \end{enumerate}
    
   \begin{remark}\label{rem:obs_handy}
        Let $A_i$ be the vertical annulus $s^{-1}(l_i)$. A properly embedded 
        horizontal disk $D \subset s^{-1}(D_i)$ intersects $A_i$ in $\alpha_i$ 
       disjoint arcs by definition of the number $\alpha_i$. Since a 
        horizontal surface $H$ intersects each typical fiber the same number of 
       times we get that $H$ must meet each fiber
        $t \cdot \lcm(\alpha_1, \ldots, \alpha_k) =t \cdot n$ times for some $t 
       \in \Z_{>0}$. If a horizontal surface meets $t\cdot n$ times a typical 
        fiber, then it meets $t \cdot n/\alpha_i$ times the special fiber $F_i$.
   \end{remark}

    \begin{lemma}\label{lem:hor_sur_coho}
       There is a bijection between $\mathcal{H}(M)$ and $HS(M):=\{\gamma 
       \in H^1(M; \Z) : \gamma([C]) \neq 0\}$ where $C$ is a generic fiber of 
       $M$.
    \end{lemma}

    \begin{proof}
        Clearly, an element $[H] \in \mathcal{H}(M)$ can be seen as the dual of 
        a $1$-form $\gamma$ with $\gamma_H(C) \neq 0$. 
        
        To see that there is a  bijection, take a 
        \hyperref[subsec:handy_model]{handy model} for $M$ 
        (we use notation  described there). Then we observe that given a 
        $\gamma \in  HS(M)$, it  restricts to a $1$-form in $H^1(M';\Z)$. The 
        manifold $M'$ is diffeomorphic to a product, so by 
        \Cref{lem:class_hor}, there  is a horizontal surface in 
        $\mathcal{H}(M')$ representing the restriction of $\gamma$ to $M'$. It 
        also  restricts as a $1$-form in 
        $H^1(s^{-1}(B); \Z)$ where we recall that 
        $s^{-1}(B)$ is a disjoint union of tori, each containing a special 
        fiber of $M$. If $\gamma([C]) =n$ then, $\gamma([F_i]) = n / \alpha_i 
        \in \Z$ so we can see the dual of $\gamma|_{s^{-1}(B)}$ as an union of 
        $n/\alpha_i$ disks in each of the tori $s^{-1}(D_i)$ for all $i$. Each 
        of these disks intersects $\alpha_i$ times the annulus $s^{-1}(l_i)$. 
        So  we can glue the horizontal surface represented by $\gamma|_{M'}$ 
        with  these disks to produce a horizontal surface in all $M$. By 
        construction, this horizontal surface represents the given $\gamma \in 
        H^1(M;\Z)$.
    \end{proof}

    \begin{lemma}\label{lem:connectivity_hor}
    Let $\hat{H}\in H(\hat{M})$ and $H:=\pi^{-1}(\hat{H})$. Then $H$  is 
    connected
    if and only if $\hat{H}$ is connected.
    \end{lemma}
    \begin{proof}
      If $H$ is connected, then so is $\hat{H}$ because $\pi$ is a 
     continuous map.
     
       Suppose now that $\hat{H}$ is connected. If $\pi^{-1}(\hat{H})$ is not 
       connected, then it is formed by parallel copies of diffeomorphic 
       horizontal surfaces. Each of them is sent by $\pi$ onto $\hat{H}$ and 
       each of them represents the same element in $HS(M)$. But, by 
       \Cref{lem:hor_sur_coho} $HS(M)$ is in bijection with $\mathcal{H}(M)$ 
       which, by \Cref{lem:11_horizontal}, is in bijection with 
       $\mathcal{H}(\hat{M})$. So we get to a contradiction.
    \end{proof}
    By construction, we have established the $1:1$ correspondences 
    
    \begin{equation}\label{eq:corresp}
        HS(M)
        \longleftrightarrow 
        \mathcal{H}(M) 
        \longleftrightarrow 
        \mathcal{H}(\hat{M})
        \longleftrightarrow 
         H^1(B; \Z) \oplus \Z  \setminus H^1(B; \Z) \oplus \{0\}
    \end{equation}
    
    Where the first correspondence is \Cref{lem:hor_sur_coho}, the second is 
    \Cref{lem:11_horizontal} and the last one is \Cref{lem:class_hor}. 
    
    Actually if we fix an orientation on the manifold and the fibers and we 
    restrict ourselves to oriented horizontal surfaces that intersect 
    positively the fibers of $M$, these are parametrized by elements in $H^1(B; 
    \Z) \oplus \Z_{>0}$. From now on we restrict ourselves 
    to  oriented horizontal surfaces $H$ with $H \cdot C >0$, that is, those 
    whose oriented intersection product with any typical fiber is positive. 
    Also the fibers are assumed to be oriented. This orientation induces a 
    monodromy on each horizontal surface.
    
    \begin{remark}\label{rem:ambiguity}
        Let $\Si$ be a surface with boundary and $\phi: \Si \rightarrow \Si$ a 
        periodic automorphism and let $\Si_\phi$ be the corresponding mapping 
        torus which is a Seifert manifold. The manifold $\Si_\phi$ fibers over 
        $S^1$ and we can see $\Si$ as a horizontal surface of $\Si_\phi$ by 
        considering any of the fibers of $f:\Si_\phi \rightarrow \MS^1$. Now 
        let 
        $\Si^{\phi}$ be the orbit space of $\Si$ which is also the base space 
        of $\Si_{\phi}$. Let  $m$ be the $\lcm$ of the multiplicities of the 
        special fibers of the Seifert fibering and let $\Si_\phi/c_m$ be the 
        quotient space resulting from the action of $c_m$ on $\Si_\phi$. We 
        observe, as before, that $\Si_\phi/c_m$ is diffeomorphic to $\Si^\phi 
        \times \MS^1$ but there is not preferred diffeomorphism between them. A
        trivialization is given by a choice of a section of $\Si_\phi /c_m \to 
        \Si^\phi$.
        
        Let $[S_1], \ldots, [S_\mu]$ be a basis of the homology group 
        $H_1(\Si^\phi;\Z)$ where each $S_i$ is a simple closed curve in 
        $\Si^{\phi}$. Let $C$ be any fiber of of $\Si_\phi/c_m$. Let $w, \hat w 
        : \Si^\phi \rightarrow \Si_\phi/c_m$ be two sections, then we have two 
        different basis of the homology of $H_1(\Si_\phi/c_m; \Z)$ induced by 
        these two sections. For instance
        $$\{[w(S_1)], \ldots, [w(S_\mu)], [C]\} \hspace{5pt} 
        \mathrm{and} \hspace{5pt} \{[\hat{w}(S_1)], \ldots, [\hat{w}(S_\mu)], 
        [C]\}.$$ Let 
        $\Si = 
        f^{-1}(0)$ be the horizontal surface that we are studying and let 
        $\hat{\Si}:= \pi (\Si)$ where $\pi$ is the quotient map $\Si_{\phi} 
        \rightarrow \Si_{\phi}/c_m$.
        Then $\hat{\Si}$ is represented with respect to the (duals of the) two 
        basis by integers $(p_1, \ldots, p_\mu, q)$ and $(\hat{p}_1, \ldots, 
        \hat{p}_\mu, q)$ respectively and $p_i \equiv \hat{p}_i \mod q$ for all 
        $i=1, \ldots, \mu$ because a section differs from another section in a 
        integer sum of fibers at the level of homology.
        
        So the numbers  $p_1, \ldots, p_{\mu}$ are well defined modulo $\Z_q$ 
        regardless of the trivialization chosen for $\Si_{\phi}/c_m$. Also by 
        the discussion above, we see that if 
        we fix a basis of $H_1(B; \Z)$, then all the elements of the form $(p_1 
        + n_1q, \ldots, p_{\mu} + n_\mu q, q)$ represent diffeomorphic 
        horizontal surfaces with the same monodromy. That there 
        exists an diffeomorphism of $M$ preserving the fibers that sends $H_1$ 
        to $H_2$ comes from the fact that on a torus $S^1 \times \MS^1$,
        there exist a diffeomorphism preserving the vertical fibers $\{t\} 
        \times \MS^1$ that sends the curve of type $(p,q)$ to the curve of type 
        $(p+kq,q)$ for any $k\in \Z$: the $k$-th power of a left handed Dehn 
        twist along some fiber $\{pt\} \times \MS^1$ that is different from $C$.
    \end{remark}

    \section{Translation Algorithms}\label{sec:algor}
    
    Every mapping torus arising from a \tat graph is a Seifert manifold so it 
    admits a (star-shaped) plumbing graph. The monodromies induced on  
    horizontal surfaces of Seifert manifolds are periodic.
     
    In this section we describe an algorithm that, given a general \tat 
    graph, produces a star-shaped plumbing graph together with the element in 
    cohomology modulo $\Z_q$ corresponding to the horizontal surface given by 
    the \tat graph. We also describe the algorithm that goes in the opposite 
    direction. 
    
    \subsection{From general \tat graph to star-shaped plumbing graph}
    
    We first state a known proposition that used in the algorithm. It can be 
    found in several references in the literature. See for 
    example \cite{Neum2} or \cite{Peder} .
    
    \begin{proposition}\label{prop:seif_to_plumb}
        Let $M(g,r; (\hat\alpha_1, \hat\beta_1), \ldots, (\hat\alpha_k, 
        \hat\beta_k))$ be a Seifert fibering. Then it is diffeomorphic as a 
        circle bundle to a Seifert fibering of the form $$M(g,r; 
        (1,b),(\alpha_1, \beta_1), \ldots, (\alpha_k, \beta_k))$$ where $0 \leq 
        \beta_i< \alpha_i$ for all $i=1,\ldots, k$. If the surface admits a 
        horizontal surface, then $b=- \sum_i \frac{\beta_i}{\alpha_i}$. The 
        corresponding plumbing graph  associated to the Seifert 
        manifold is
        \begin{figure}[H]
            \centering
            \includegraphics[scale=0.7]{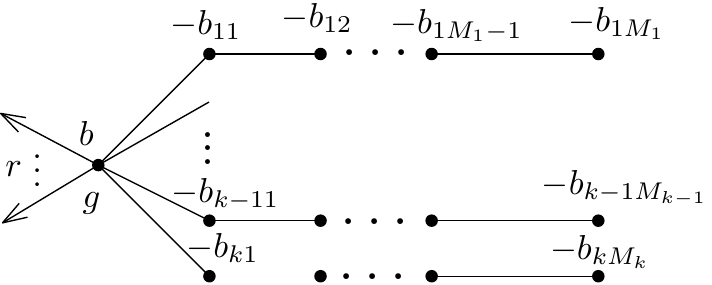}
            \caption{Plumbing graph for a Seifert manifold}
            \label{fig:seifplum}
        \end{figure}
        
        where the numbers $b_{ij}$ are the continuous fraction expansion for 
        $\alpha_i/ 
        \beta_i$, that is $$\frac{\alpha_i}{\beta_i}= b_{i1} - \frac{1}{b_{i2}- 
        \frac{1}{b_{i3}-
                \frac{1}{\ldots}}}$$
    \end{proposition}

    Let $(\G, \mathcal{P}, \sigma)$  be a general \tat structure. For 
    simplicity, we
    suppose that $\G$ is connected. This includes as particular cases pure \tat
    graphs and relative \tat graphs. Let $\phi_\G$ be a truly periodic
    representative of the \tat automorphism and let $\Si_{\phi_\G}$ be the 
    mapping torus of the the diffeomorphism $\phi_\G:\Si \rightarrow \Si$. The 
    mapping torus given by a periodic diffeomorphism of a surface is a Seifert 
    manifold.  We describe an algorithm that takes $(\G, \mathcal{P}, \sigma)$ 
    as input and returns as output:
    
    \begin{enumerate}
        \item The invariants of a Seifert manifold: 
        $$M(g,r; 
        (1,b), (\alpha_1, \beta_1), \ldots, (\alpha_k,\beta_k))$$
        diffeomorphic to the mapping torus $\Si_{\phi_\G}$. It is represented 
        by a star-shaped 
        plumbing graph $\Lambda$ 
        corresponding to the Seifert manifold and
        
        \item  a tuple $(p_1, 
        \ldots, p_{\mu},q)$ with $p_i 
        \in \Z/q\Z$ and $q \in \Z_{>0}$ representing the horizontal surface 
        given by $\Si$ with respect to some basis of the homology $H_1(B ; \Z) 
        \simeq \Z^\mu$ of the base space $B$ of $M$.
    \end{enumerate}

    \textbf{Step 1.} We consider $\G^{\phi_\G}$, that is the quotient space
    $\G/\sim$ where $\sim$ is the equivalence relation induced by the action of 
    the safe walks on the graph. This graph is nothing but the image of $\G$ by 
    the projection of the branched cover $p: \Si \rightarrow \Si^{\phi_\G}$ 
    onto the orbit space.
    
    The map $p|_{\G}:\G \rightarrow \G^{\phi_\G}$ induces a ribbon graph 
    structure on $\G^{\phi_\G}$. We can easily get the genus $g$ and number of 
    boundary components $r$ of the thickening of $\G^{\phi_\G}$ from the 
    combinatorics of the graph. This gives us the first two invariants of the 
    Seifert manifold.
    
    \textbf{Step 2.}  Let $sv(\G)$ be the set  of points with non trivial 
    isotropy subgroup in $<\phi_\G>$.  This is the set of branch points of 
    $p:\Si \rightarrow \Si^{\phi_\G}$ by definition. 
    
    Let $v \in sv(\G)$. Then there exists $m<n$ with $n=m \cdot s$ such that 
    $v$ is a fixed point of $\phi_\G^{m}$ (take $m$ the smallest natural number 
    satsifying that property). We can therefore use 
    \Cref{cor:rotation_seifertinv}. We get  that
    $\phi_\G^m$ acts as rotation with rotation number $p/s$ in a small disk 
    centered at $v$. So  around the fiber corresponding to the vertex 
    $v$, the  Seifert manifold is diffeomorphic fiberwise to a ${p,s}$-torus. 
    and the corresponding Seifert pair $(\alpha_v, \beta_v)$  is given by 
    $(\alpha_v, \beta_v) = (s,-b)$ with $bp \equiv 1 \mod q$. 
    
    We do this for every vertex in $sv(\G)$ and we get all the Seifert pairs.
    
    
    \textbf{Step 3.} Since we have already found a complete set of Seifert
    invariants, we have that $$M(g,r;(\alpha_1, \beta_1), 
    \ldots,
    (\alpha_k, \beta_k),(\tilde{\alpha}_{k+1}, \tilde{\beta}_{k+1}))$$ is 
    diffeomorphic to 
    the
    mapping torus of the pair $(\Si, \phi_\G)$. 
    
    Now we use \Cref{prop:seif_to_plumb} to get the normalized form of the 
    plumbing graph associated to the Seifert manifold. 
    
    \textbf{Step 4.} For this step, recall \Cref{sec:hor_sur} and notation
    introduced there. Fix a basis $[S_1^1], \ldots, [S_{\mu}^1]$ of $H_1(B; \Z)$
    where $S_i^1$ is a simple closed curve contained in $\G^{\phi_{\G}}$.
    
    Let $m := \lcm(\alpha_1, \ldots, \alpha_k)$. Observe that necessarily $m|n$ 
    so $n= m \cdot q$ (with $n$ the order of $\phi_\G$).
    
    This number $q$ that we have found is the last term of the cohomology 
    element we are looking for. Of course, it is also the oriented intersection 
    number of  $\hat{\G}=\pi(\G)$ with $C$.  (recall $\pi$ was the projection 
    $M  \rightarrow M/c_m=:\hat{M}$).
    
    Pick any basis of $H_1(B; \Z)$ by picking a collection of circles $S_1^{1},
    \ldots, S_k^1$ contained in the orbit graph $\G^{\phi_\G} \subset B$ that
    generate the homology of the graph. 
    
    Now if we intersect the graph $\G/ {c_m} \cap S_i^1 \times C $ with the 
    torus over one of the representatives of the basis, we get a collection of 
    $k_i$ closed curves, each one isotopic to the curve of slope $p_i'/q_i'$ 
    where  $q_i' \cdot k_i = q$ and $p_i = p_i' \cdot k_i = p_i$. This number, 
    $p_i$, is the $i-th$ coordinate of the cohomology element with respect to 
    the fixed basis.
    
    We can compute $p_1, \ldots, p_k$ directly. Let $S_i^1$ be one of the 
    generators of $H_1(\G^{\phi_\G}; \Z)$. Let $\hat{S}_i^1 
    := \G /c_m  \cap S_i^1 \times \MS^1 $  where $ S_i^1 \times \MS^1 \subset 
    \G^{\phi_\G} \times  \MS^1$; and let $\tilde{S}_i^1 := 
    p|_{\G}^{-1}(\hat{S}_i^1)$. Observe that $\hat{S}_i^1$
    consists of $k_i$ disjoint circles and that $k_i$ divides $q$. Let $q_i' =
    q/k_i$.

    Pick a point $z \in S_i^1$ which is not in the image by $\pi$ of a special
    fiber. Then $\hat{\pi}^{-1}(z) \cap \G/ c_m$ consists of $q$ points lying 
    in the $k_i$ connected components of $\hat{S}_i^1$. Pick one of these 
    connected components and enumerate the corresponding $q_i'$ points in it 
    using the
    orientation induced on that connected component by the given orientation of
    $S_i^1$. Then we have the points $z_1, \ldots, z_{q_i'}$. We observe that by
    construction, these points lie on the same fiber in $\hat{M}$ and this 
    fiber is
    oriented. Follow the fiber from $z_1$ in the direction indicated by the
    orientation, the next point is $z_{t_i}$, with $t_i \in \{1, \ldots, 
    q_i'\}$. We
    therefore find that this connected component of $\hat{S}_i^1$ lies in $S_i^1
    \times \MS^1$ as the curve with slope $(t_i-1)/q_i'$ and so $p_i=(t_i -1) 
    \cdot
    k_i$.

    \subsection{From star-shaped plumbing graph to \tat graphs.}

    The input that we have is:
    \begin{enumerate}
        \item[i)] A Seifert fibering of a manifold $M$.
        \item[ii)] A horizontal surface given by an element in $H^1(B \times 
        \MS^1; \Z)$ that does not vanish on a typical Seifert fiber.
    \end{enumerate}
    
    The output is:
    
    \begin{enumerate}
        \item A general, relative or pure \tat graph such the induced mapping 
        toru is  diffeomorphic to the given plumbing manifold in the input. And 
        such  that the thickening of the graph, represents the horizontal 
        surface given.
    \end{enumerate}
    
    \textbf{Step 1.}We start with a Seifert fibering $M(g,r;(\alpha_1,\beta_1, 
    \ldots, \alpha_k, \beta_k)$. 
    
    We fix a model of the Seifert fibering as in  \nameref{subsec:handy_model}. 
    We
    recall that the model consists of the following data:
    \begin{enumerate}
        \item[i)] The Seifert fibering $s:M \rightarrow B$ where $B$ is a 
        surface of
        genus $g$ and $r$ boundary components.
        \item[ii)] A collection of arcs $\{l_i\}$ with $i=1, \ldots, k$ properly
        embedded in $B$ where the boundary of these arcs lie in one chosen 
        boundary
        component of $B$. These satisfy that when we cut along one of them, say 
        $l_i$ we
        cut off a disk denoted by $D_i$ from $B$ that contains the image of 
        exactly one
        special fiber, we denote the image of this fiber by $x_i$.
        \item[iii)] We have an identification of each solid torus $s^{-1}(D_i)$ 
        with the
        corresponding fibered solid torus $T_{p_i,q_i}$ with $q_i= \alpha_i$ and
        $-p_i\beta_i \equiv 1 \mod \alpha_i$ and $0<p_i<q_i$.
    \end{enumerate}
    
    \textbf{Step 2.} Observe that $B$ is homotopic to a wedge of 
    $\mu = 2g-r+1$ circles that does not contain any $x_i$ for $i=1, \ldots, 
    i$. We can see this wedge as a spine embedded in $B$. Denote by $c$ the 
    common point of all the circles. Now we embed disjoint
    segments $e_i$ with $i=1, \ldots, k$ where each one satisfies that one of 
    its ends lies in the  spine and the other end lies in $x_i$. Also, they 
    do not intersect the wedge of circles at any other point and they also 
    $e_i$ does not intersect any $D_j$ for $j \neq i$. 
    We denote the union of the wedge and these segments by $\tilde{\Lambda}$ and
    observe that $\tilde{\Lambda}$ is a spine of $B$.
    
    \textbf{Step 3.}
    We suppose that the element in $H^1(\hat{M}; \Z)$ given is irreducible,
    otherwise if it is of the form $k(p_1', \ldots, p_\mu', q')$ with $(p_1',
    \ldots, p_\mu', q')$ irreducible, we take the irreducible part, carry out 
    the following construction of the corresponding horizontal surface and then 
    take $k$ parallel copies of this surface.
    
    Recall \Cref{diag:seif_orbit} for the definition of the maps $s, \hat{s}$ 
    and $\pi$.
    
    Once and for all, fix a trivialization $\hat{M} \simeq B \times \MS^1$. We 
    assume that the element $(p_1, \ldots, p_\mu, q) \in
    H^1(\hat{M}; \Z)$ is expressed with respect to the dual basis  $[S_1] , 
    \ldots, [S_{\mu}], [C] $ where  the first $\mu$ are circles of the wedge 
    embedded in  $B$  and $[C]$ is the homology class of $C:=\hat{s}^{-1}(c)$. 
    
    For each $i= 1, \ldots, \mu$, consider the torus $\hat{s}^{-1}(S_i)$ which 
    is naturally trivialized by the trivialization of $\hat{M}$.  We pick in it 
    $k_i$ copies of the curve of slope $p_i'/q'$ where $p_i/q = k_ip_i'/k_iq'$. 
    For each $i$, the curves constructed this way in $\hat{s}^{-1}(S_i)$ 
    intersect $q$ times the curve $C$. Hence we can isotope them so that all of 
    them intersect $C$  in the same $q$ points. We denote the union of these 
    curves by  $\hat{\Lambda}'$. We assume that $\hat{\Lambda}'$ projects to 
    $\tilde\Lambda \setminus \bigcup e_i$ by $B \times \MS^1 \to B$.
    
    By construction, $\hat{\Lambda}'$ is a ribbon graph for the surface 
    horizontal surface $\hat{H} \subset \hat{M}$. Observe that  
    $s(\hat{\Lambda}') \neq  \tilde{\Lambda}$. However $s(\hat{\Lambda}')$ is 
    also a spine of $B$ (it coincides with the wedge of circles in $B$).
    
    Define $\Lambda':=\pi^{-1}(\hat{\Lambda}')$. By the definition of $\pi$, 
    this graph can also be constructed by taking in each of the tori 
    $\pi^{-1}(\hat{s}^{-1}(S_i))=s^{-1}(S_i)$ , $k_i$ copies of the curve of 
    slope $p_i'/n$. Which by construction all intersect in $n$ points in 
    $s^{-1}(c)$.
    
    \textbf{Step 4.} Now we describe $\pi^{-1}(\hat{s}|_{\hat{H}}^{-1}(e_i))$ 
    for 
    each $i=1, \ldots, 
    k$.  First we observe that it is equal to $s|_{H}^{-1}(e_i)$ which  is a 
    collection of  $q\cdot n/\alpha_i$ disjoint start shaped graphs. Each 
    star-shaped 
    piece has $\alpha_i$. To find out the gluings of these arms with $\Lambda'$ 
    one looks as the structure of $s^{-1}(D_i)$ as a $(c,\alpha_i)$- solid 
    torus. To visualize it, place the $q\cdot n/\alpha_i$ star-shaped pieces in 
    a solid cylinder $D^2 \times [0,1]$ and identify top with botton by a 
    $c/\alpha_i)$ rotation. The fibers of the fibered torus give the monodromy 
    on the end of the arms and the attaching to $\Lambda'$.
    
    We define $\Lambda$ as the union of $\Lambda'$ with these star-shaped pieces
    
    \textbf{Step 5.} The embedding of $H$ in the Seifert 
    manifold defines a diffeomorphism $\phi: H \rightarrow H$ in the following 
    way. Let  $x\in H$ and follow the only fiber of the Seifert manifold that 
    passes through $x$ in  the direction indicated by its orientation, 
    we define $\phi(x)$ as the  next point of intersection of that fiber with 
    $H$.
    
    To describe $\phi$ up to isotopy it is enough to give the rotation numbers 
    of  $\phi$ around each boundary component of $H$ plus some spine invariant
    by $\phi$. By construction $\Lambda$ is an invariant
    graph. The fibers of the Seifert fibering give us an automorphism on the 
    graph $\Lambda$. To get the rotation numbers, we cut the thickening $H$ 
    along $\Lambda$ and we get a collection of cylinders $\Lambda_j \times 
    [0,1]$ with $j=1, \ldots, r'$.

    Now we invoke \cite[Theorem 5.12]{jmp} if the monodromy leaves at least $1$ 
    boundary component invariant and we invoke \Cref{thm:general_tat} if the 
    monodromy does not leave any boundary component invariant. This gives us
    a constructive method to find a graph (which in general will be different 
    from $\bigcup_{i=1}^\mu S_i \bigcup_{j=1}^k e_j$) containing all  branch 
    points in $B$  such that it is a retract of $B$ and such that it admits a 
    metric that makes its preimage a \tat graph. 

    \section{Examples}\label{sec:examples}
    
    We apply the algorithms developed in the previous sections to two examples.
    
        \begin{example}
        Suppose we are given the bipartite complete graph $\G$ of type 
        $4,11$ with the cyclic ordering induced by placing $4$ and $11$ 
        vertices in two horizontal parallel lines in the plane and taking the 
        joint of the two sets in that plane. Give each edge length $\pi/2$. 
        This metric makes it into a \tat graph as we already know. Let 
        $\phi_\G$ be a periodic representative of the mapping class
        induced by the \tat structure.
        
        \begin{figure}[!ht]
            \centering
            \includegraphics[scale=0.7]{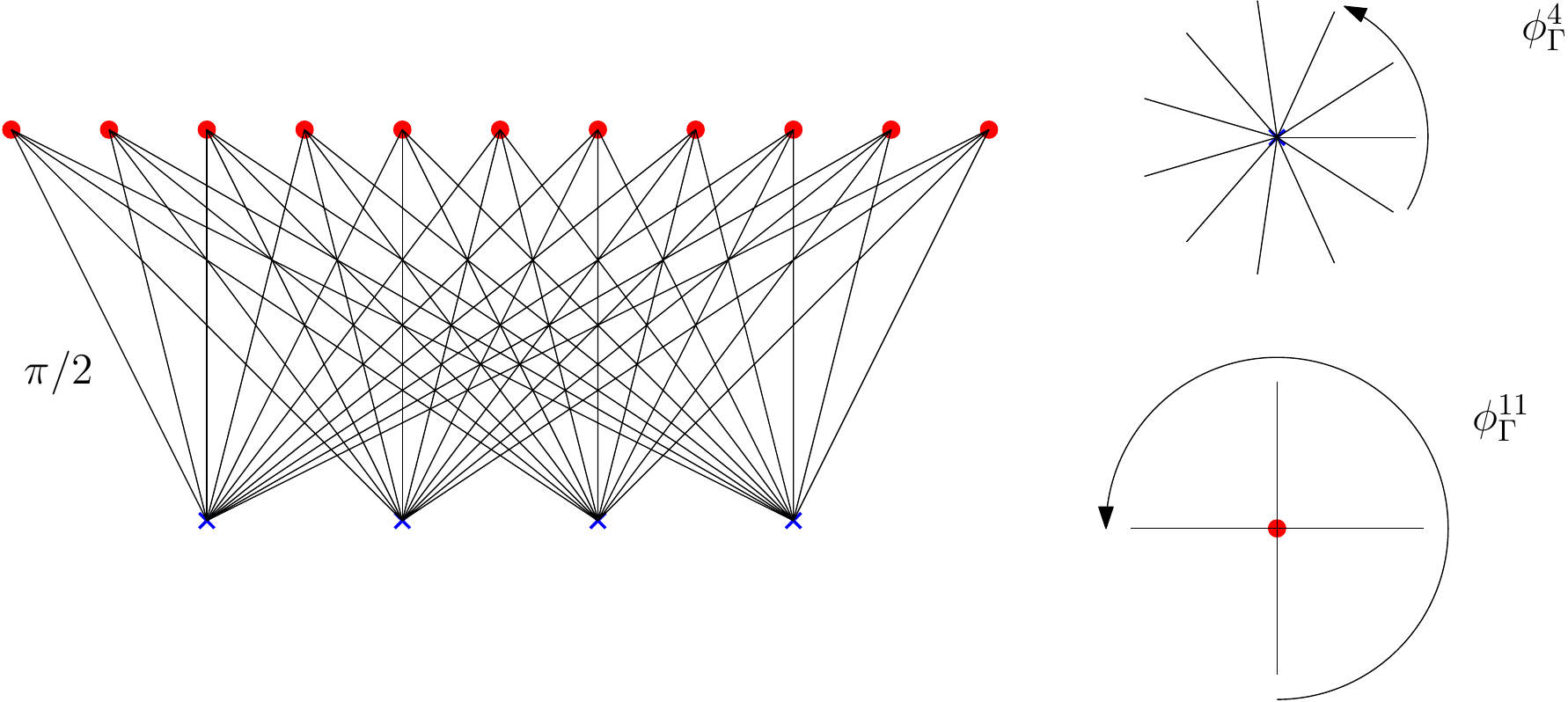}
            \caption{On the left we see the \tat graph $K_{4,11}$. On the right 
                we see a
                small neighbourhood of a vertex of valency $11$ where 
                $\phi_\G^4$ acts as the
                rotation $r_{4/11}$ radians. Equivalently, for a vertex of 
                valency $4$, we see
                that $\phi_\G^{11}$ acts as the rotation $r_{3/4}$.}
            \label{fig:k411}
        \end{figure}
        
        Let's find the associated invariants. One can easily check that the 
        orbit graph is just a segment joining the only two branch 
        points
        so the orbit surface is a disk and hence $g=0$ and $r=1$. 
        
        The map $p:\Si  \rightarrow \Si^{\phi_\G}$ has two branch points that 
        correspond to two Seifert pairs. Let $r_1$ be the branching point in 
        which preimage lie the $4$ points of valency $11$. We choose any of 
        those $4$  points and denote it $p_1$, now $\phi^4$ acts as a rotation 
        with rotation  number $4/11$
        in a small disk around $p_1$ . Hence, the associated normalized Seifert 
        pair is $(11,8)$. 
        Note that $8 \cdot  4 \equiv -1 \mod 11$ and that $0<8 <11$. 
        Equivalently for the other point we 
        find that $\phi_\G^{11}$ is a rotation with rotation number $3/4$ when 
        restricted  to a disk around any of the $11$ vertices of valency $4$. 
        Hence, the  corresponding  normalized Seifert pair is $(4,1)$.
        
        Computing the continued fraction we have that $\frac{11}{8}=[2,2,3,2]$ 
        and $\frac{4}{1}=[4]$. For computing the number $b$ we think of the 
        surface resulting from extending the periodic automorphism to a disk 
        capping off the only boundary component of $\Si$. By a similar 
        argument, since the rotation number induced on the boundary is $-1/44$, 
        this would lead to a new  Seifert pair $(44,1)$. Since these are 
        normalized Seifert invariants, the new  manifold is closed and admits a 
        horizontal surface, we can use  \Cref{prop:seif_to_plumb} and compute 
        the number $b$ as $- 1/4 - 8/11 - 1/44 = -1$.
        
        So the plumbing diagram corresponding to the mapping torus of $\Si$ by 
        $\phi_\G$
        is the following.
        
        \begin{figure}[H]
            \centering
            \includegraphics[scale=0.7]{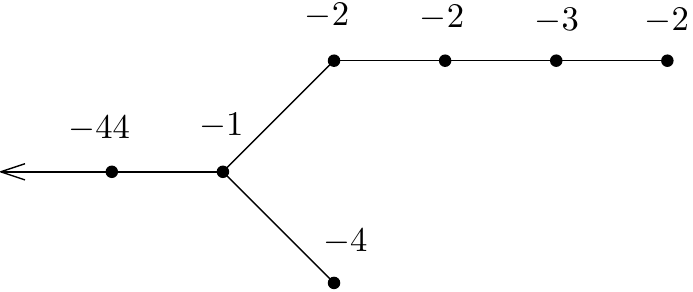}
            \label{fig:plumb0}
        \end{figure}

        which, up to contracting the bamboo that ends in the arrowhead, 
        coincides with the dual graph of the resolution of the 
        singularity of $x^4+y^{11}$ at $0$.
        
        Finally, we are going to compute the element that the surface $\Si$ 
        represents in the homology group $H_1(\Si^{\phi_\G}) \oplus \Z$ . First 
        observe that since $\Si^{\phi_\G}$ is a disk, the group is isomorphic 
        to $0 \oplus \Z$. This tells us that the only possible choices of 
        multisections  in the bundle  $\Si^{\phi_{\G}} \times \MS^1$ are 
        classified (up to isotopy) by the  elements $(0,k)$ with $k \neq 0$. 
        The element $(0,k)$ corresponds to $k$ parallel copies
        of the disk $\Si^{\phi_\G}$. In our case, there is only one such disk 
        so the element is $(0,1)$.
    \end{example}

    \begin{example}\label{ex:algorithm}
        Suppose we are given the following plumbing graph:
        
        \begin{figure}[H]
            \centering
            \includegraphics[scale=0.7]{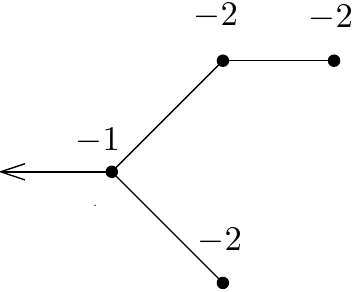}
            \caption{}
            \label{fig:plumb2}
        \end{figure}
        
        We are indicated two of the invariants of the Seifert manifold: the 
        genus of the
        base space $g=0$ and its number of boundary components $r=2$. The base 
        space $B$
        is therefore an annulus. 
        
        We compute the Seifert invariants by interpreting the weights on the 
        two bamboos
        of the plumbing graph as numbers describing continued fractions. We get 
        $[2,2]=
        3/2$ and $[2]=2$ so the Seifert pairs are $(3,2)$ and $(2,1)$. So the
        corresponding Seifert fibering $s:M \rightarrow B$ has two special 
        fibers $F_1$
        (for the pair $(2,1)$ and $F_2$ (for the pair $(3,2)$). Using 
        \Cref{lem:seifpq}
        we have that the Seifert fiber corresponding to the pair $(3,2)$ has a 
        tubular
        neighborhood diffeomorphic to the fibered solid torus $T_{1,3}$; this 
        is because
        $-2 \cdot 1 \equiv 1 \mod 3$. Analogously, the fiber corresponding to 
        the
        Seifert pair $(2,1)$ has a tubular neighborhood diffeomorphic to the 
        fibered
        solid torus $T_{1,2}$.
        
        Now we fix a model for our Seifert manifold. Take an annulus as in
        \Cref{fig:base_ex}. Now we use the kind of model explained in
        \Cref{fig:base_seif}; we choose a boundary component and we pick 
        properly
        embedded arcs (with their boundaries lying on the chosen boundary 
        component) in
        such a way that cutting along one of them cuts off a disk containing 
        only one of
        the two images by $s$ of the special fibers; over those disks in $M$ 
        lie the two
        corresponding fibered solid tori. Let $d$ be the point lying under 
        $\pi(F_1)$
        and let $a$ be the point lying under the fiber $\pi(F_2)$. We pick an 
        embedded
        graph which is a spine of $B$ as in \Cref{fig:base_ex} below, that is, 
        the graph
        is the union of: a circle whose class generates the homology of the 
        base space.
        We denote it by $S$;  a segment joining a point $c \in S$ with the 
        vertex $d$.
        We denote this segment by $D$ and a segment joining a point $b \in S$ 
        with the
        vertex $a$. We denote this segment by $A$. See \Cref{fig:base_ex}.
        
        We denote this graph by $\tilde{\Lambda}$.
        
        \begin{figure}[H]
            \centering
            \includegraphics[scale=0.5]{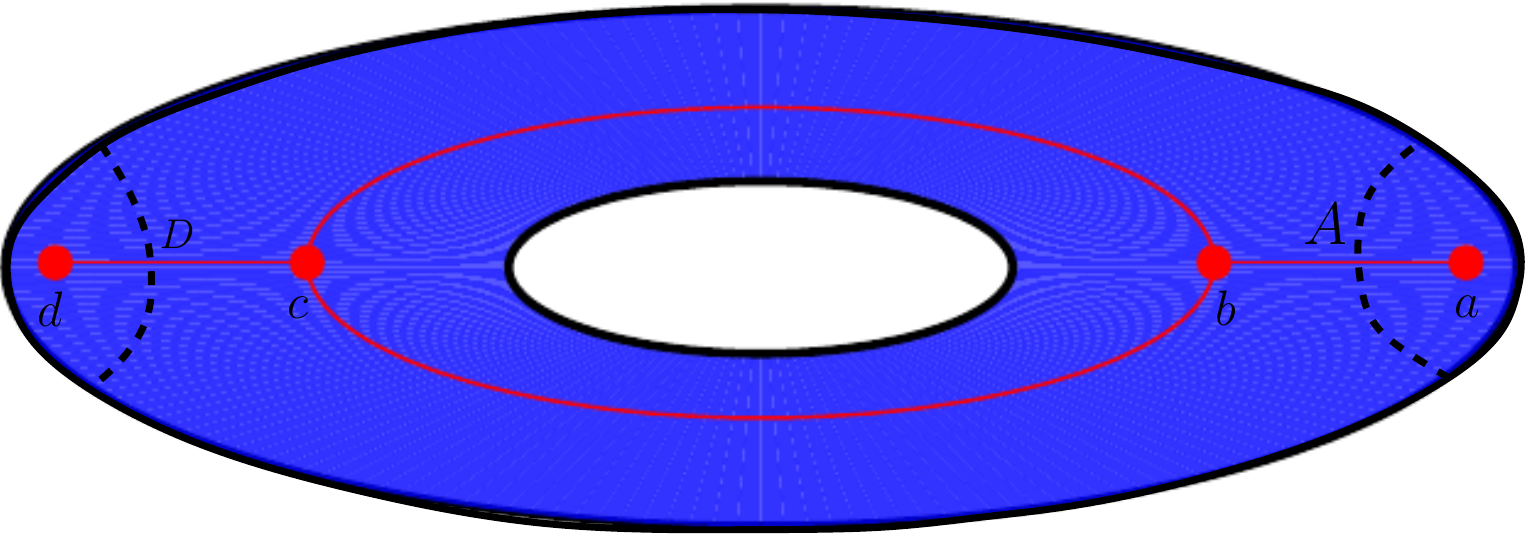}
            \caption{This is the base space $B$ of the Seifert fibering. In red 
            we see
                $\tilde{\G}$ which is formed by a circle and two segments 
                attached to it that
                end at the image by $s$ of the special fibers. The dashed lines 
                represents the
                properly embedded arcs }
            \label{fig:base_ex}
        \end{figure}
        
        Now we consider $\hat{M}$ which is diffeomorphic to $ B \times \MS^1$ 
        which is
        homotopically equivalent to $\tilde{\G} \times \MS^1$. We denote the 
        projection on
        $B$ by $\hat s:\hat{M} \rightarrow B$. The map $\pi: M \rightarrow 
        \hat{M}$
        satisfies that $\hat{s} \circ \pi = s$.
        
        The piece of information missing from the input is the horizontal 
        surface.
        Suppose we are given the element $(1,2) \in H^1(B; \Z) \oplus \Z$ with 
        respect
        to the basis formed by the class of $S$.  Then, the intersection of the
        horizontal surface $\hat{H} \subset \hat{M}$ with the torus $\hat S
        :=\hat{s}^{-1}(S)$ is a curve of slope $1/2$. We also have that
        $\hat{s}^{-1}(A)$ consists of two segments, as well as 
        $\hat{s}^{-1}(D)$.  See
        figure \Cref{fig:hat_m}.
        
        \begin{figure}[h]
            \includegraphics[scale=0.6]{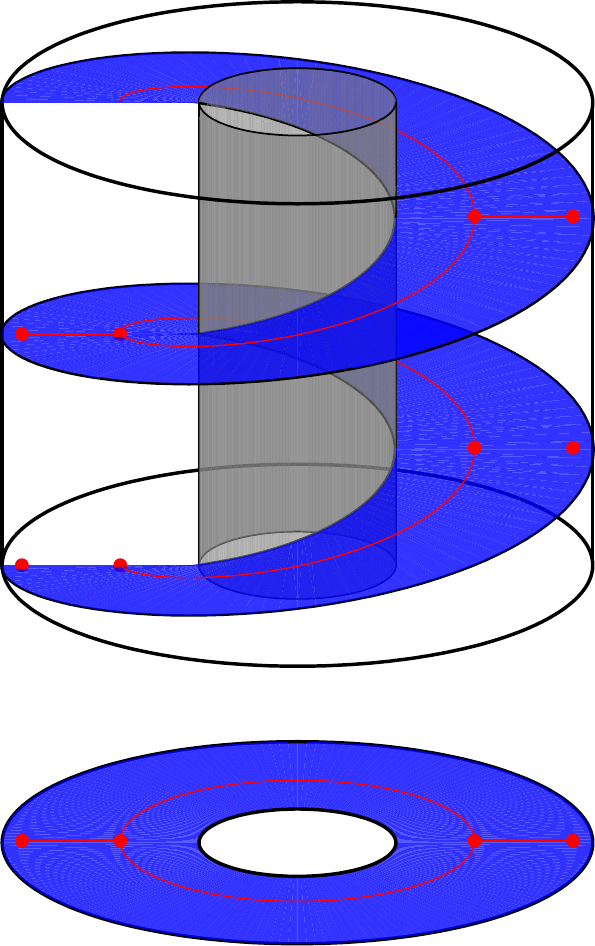}
            \caption{This is $\hat{M}$ together with the base space under it. 
            Lying over the graph $\tilde{\G}$ we can see the graph $\hat{G}$ 
            whose thickening is the  horizontal surface $\hat{H}$ (the blue 
            helicoidal ramp on the  figure). Also we see that lying over the 
            circle of $\tilde{\G}$ lies the closed  curve in  $\hat{\G}$ that 
            is a curve of slope $1/2$ in the torus  $\hat{s}^{-1}(S_1)$.}
            \label{fig:hat_m}
        \end{figure}
        
        The horizontal surface that we are looking for  is 
        $H:=\pi^{-1}(\hat{H})$ that
        is the thickening of $\pi^{-1}(\hat{\G})$. To know the topology of $H$ 
        and the
        action on it of the monodromy, we  construct the ribbon graph
        $\pi^{-1}(\hat{\G})$. We observe that $\lcm(2,3)=6$ so 
        $\pi^{-1}(\hat{S})$ is
        the curve of slope $1/12$ on the torus $s^{-1}(S)$. We also have that 
        $s^{-1}(a)
        = (\pi \circ \hat s)^{-1}(a)$ consists of $4$ and $s^{-1}(A)$ consists 
        of $12$
        segments separated in groups giving valency $3$ to each of the points 
        in 
        points. Equivalently $s^{-1}(d)$ consists of $6$ vertices and 
        $s^{-1}(D)$ of
        $12$ segments naturally separated by pairs. The fact that $s^{-1}(S)$ 
        is a curve
        of slope $1/12$, give us the combinatorics of the graph. Using notation 
        of
        \ref{fig:aux_graph}, and the rotation numbers associated to each of the 
        two
        Seifert pairs, we have that the graph is that of \Cref{fig:aux_graph}.
        
        \begin{figure}[h]
            \centering
            \includegraphics[scale=0.5]{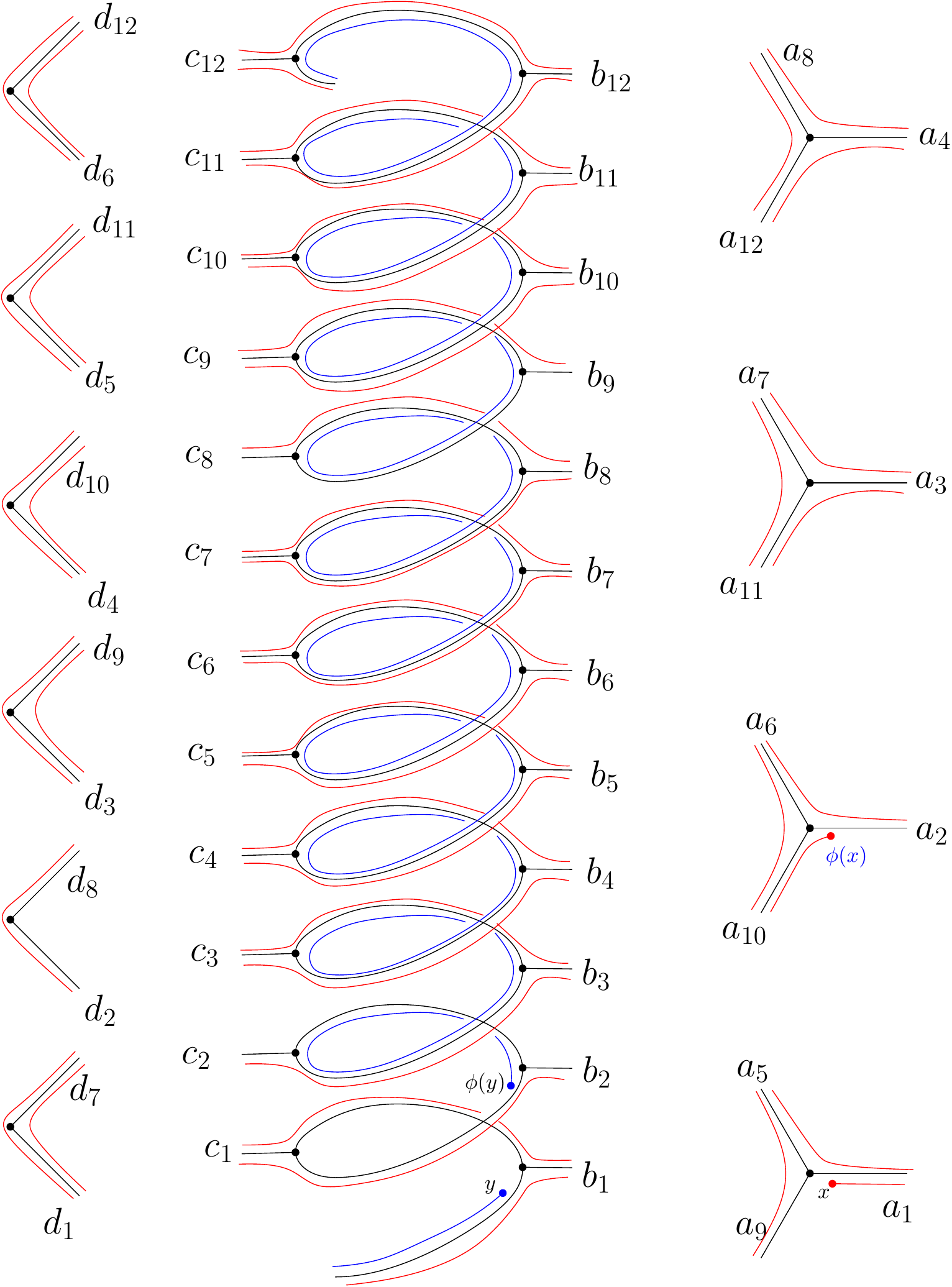}
            \caption{The graph $\pi^{-1}(\hat{\Lambda})$ in black. The letters 
            with
                subindexes are interpreted like this:  the vertex $a_i$ is 
                glued to the vertex
                $b_i$ and the vertex $d_i$ is glued to the vertex $d_i$. In red 
                we see a path
                from $x$ to $\phi(x)$ used to compute the rotation number of 
                $\phi$ with respect
                to the outer boundary component; we observe that the outer 
                boundary component
                retracts to $72$ edges (each edge is counted twice if the 
                boundary component
                retracts to both sides of the edge), and the red path covers 
                $66$ of these
                edges.}
            \label{fig:aux_graph}
        \end{figure}
        
        You can easily compute from the ribbon graph that the surface has $2$ 
        boundary
        components and genus $7$. Since it has only two boundary components, 
        each of
        them is invariant by the action of the monodromy induced by the 
        orientation on
        the fibers. We compute their rotation numbers as explained on Step $5$ 
        of the
        algorithms. We observe that the ''outer'' boundary component retracts 
        to $72$
        edges where an edge is counted twice if the boundary component retracts 
        to both
        sides of it. We pick a point $x$ and observe that the monodromy 
        indicated by the
        orientation of the fibers takes it to the point inmediately above it 
        $\phi(x)$.
        Now we consider a path ''turning right'' starting at $x$ and observe 
        that it
        goes along $66$ edges before reaching $\phi(x)$. Hence, the rotation 
        number of
        $\phi$ with respect to this boundary component is $11/12$. Similarly, 
        we observe
        that the other boundary component retracts to $24$ edges and by a 
        similar
        procedure we can check that $\phi$ also has a rotation number $11/12$ 
        with
        respect to this other boundary component. See figure 
        \Cref{fig:aux_graph}.
        
        Following the construction in \cite[Theorem 5.12]{jmp}, we should put a 
        metric on
        $\hat{G}$ so that the part where the outer boundary component retracts 
        has a length of $\pi/11$ and the same for the other boundary component. 
        But this is impossible given the combinatorics of the graph. That means 
        that this  graph does not accept a \tat metric. However theorem 
        \cite[Theorem 5.12]{jmp} gives us a procedure to find a graphadmittin a 
        \tat metric  producing the given monodromy. In this case, 
        it is enough to consider the following  graph.
        
        \begin{figure}[H]
            \centering
            \includegraphics[scale=0.5]{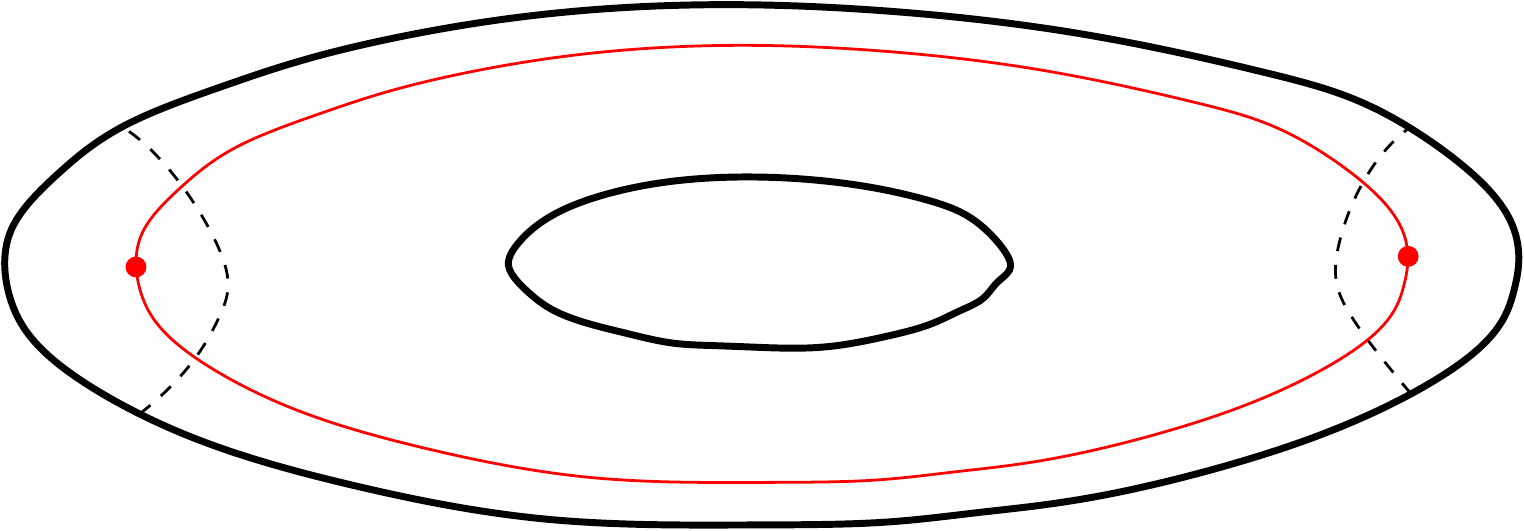}
            \caption{Graph $\tilde{\G}$ that admits a \tat metric.}
            \label{fig:base_ex_def}
        \end{figure}
        
        If we call this graph $\tilde{\G}$ we see that that 
        $\G:=p^{-1}(\tilde{\G})$ is
        the following graph:
        
        \begin{figure}[H]
            \centering
            \includegraphics[scale=0.5]{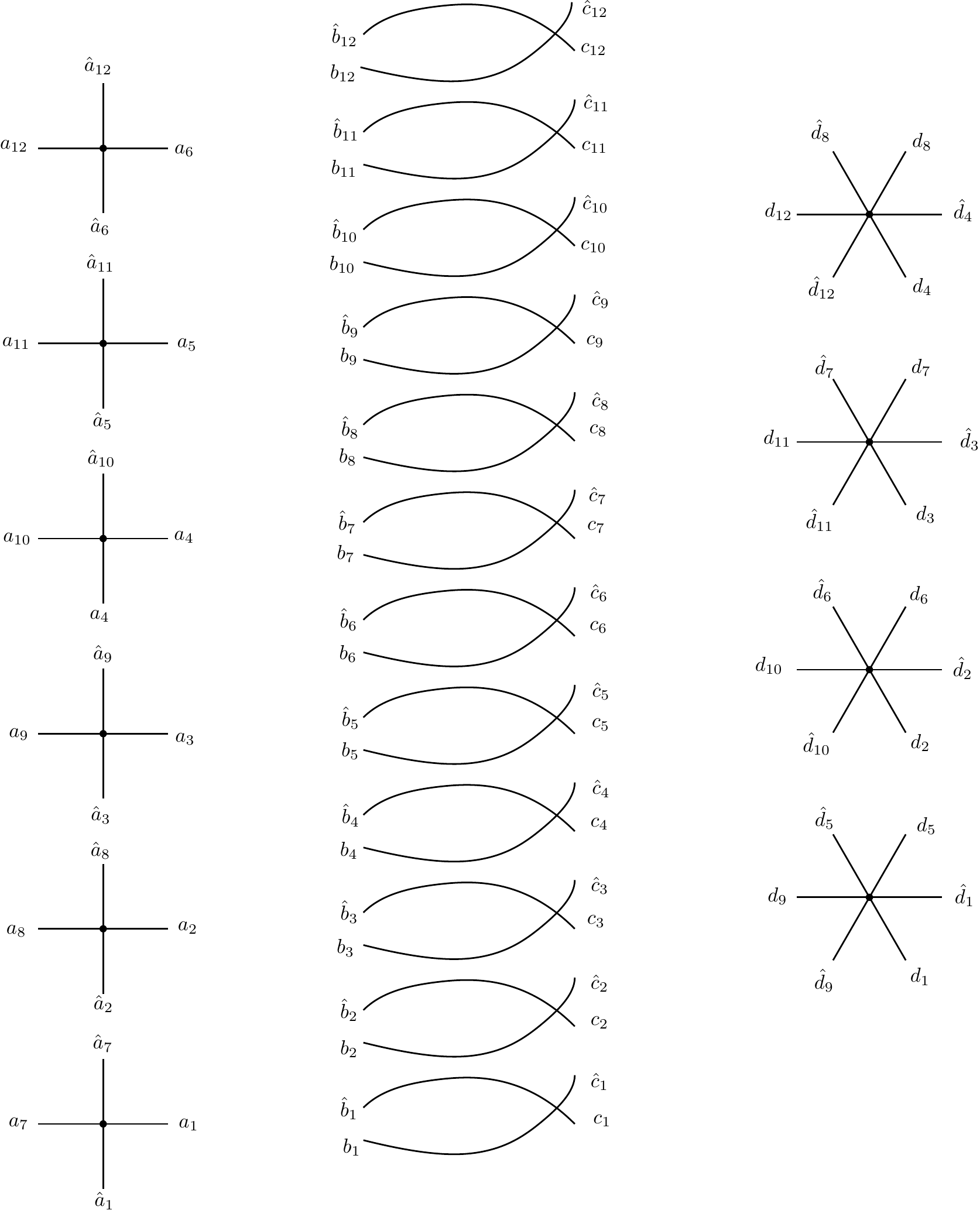}
            \caption{The \tat graph $\G$. The notation means that $a_i$ is 
            glued to $b_i$,
                $\hat{a}_i$ to $\hat{b}_i$, $c_i$ to $d_i$ and $\hat{c}_i$ to 
                $\hat{d}_i$ for
                all $i=1, \ldots, 12$.}
            \label{fig:tat_graph}
        \end{figure}
    
    By setting each 
    of the two edges of the circle $\tilde{\G}$ has length $\pi/22$, then $\G$ 
    is a  pure \tat graph modelling the action of the monodromy on the 
    horizontal surface  $H$.
    \end{example}
    \bibliographystyle{alpha}
    \bibliography{bibliography}

\end{document}